\documentclass[11pt]{amsart}

\usepackage{graphicx}%
\usepackage{microtype}
\usepackage{multirow}%
\usepackage{amsmath,amssymb,amsfonts}%
\usepackage{stmaryrd}
\SetSymbolFont{stmry}{bold}{U}{stmry}{m}{n}
\DeclareMathSizes{6}{6}{5}{5}
\usepackage{amsthm}%
\usepackage{caption}
\newcommand{\R}{\mathbb{R}}

\newcommand{\CA}{\mathcal{A}}
\newcommand{\BB}{\mathcal{B}}
\newcommand{\CC}{\mathcal{C}}
\newcommand{\DD}{\mathcal{D}}

\newcommand{\GG}{\mathcal{G}}

\newcommand{\II}{\mathcal{I}}
\newcommand{\LL}{\mathcal{L}}
\newcommand{\NN}{\mathcal{N}}
\newcommand{\PP}{\mathcal{P}}
\newcommand{\QQ}{\mathcal{Q}}

\newcommand{\CU}{\mathcal{U}}
\newcommand{\TT}{\mathcal{T}}

\newcommand{\ff}{\mathbf{f}}
\newcommand{\kk}{\mathbf{k}}

\newcommand{\nn}{\mathbf{n}}

\newcommand{\uu}{\mathbf{u}}
\newcommand{\vv}{\mathbf{v}}

\newcommand{\xx}{\mathbf{x}}
\newcommand{\yy}{\mathbf{y}}
\newcommand{\zz}{\mathbf{z}}

\newcommand{\ess}{\mathbf{s}}
\renewcommand{\O}{{\mathbf O}}
\newcommand{\0}{\mathbf{0}}

\newcommand{\ga}{\alpha}
\newcommand{\gb}{\beta}
\newcommand{\gc}{\gamma}
\newcommand{\gd}{\delta}
\newcommand{\vf}{\varphi}
\newcommand{\gk}{\kappa}
\newcommand{\GD}{\Delta}
\newcommand{\ve}{\varepsilon}
\newcommand{\GO}{\Omega}
\newcommand{\go}{\omega}
 \newcommand{\gl}{\lambda}
\newcommand{\gth}{\theta}
\newcommand{\GC}{\Gamma}
\newcommand{\gz}{\zeta}

\newcommand{\GS}{\Sigma}

\newcommand{\gtt}{\boldsymbol{\tau}}

\newcommand{\nnu}{\boldsymbol{\nu}}

\newcommand{\ph}{\hat{p}}

\newcommand{\sh}{\hat{s}}
\newcommand{\vh}{\hat{v}}
\newcommand{\wh}{\hat{w}}
\newcommand{\rhoh}{\hat{\rho}}

\newcommand{\db}{\bar{d}}

\newcommand{\rb}{\bar{r}}
\newcommand{\Ub}{\overline{\CU}}
\newcommand{\vb}{\bar{v}}

\newcommand{\gbb}{\bar{\beta}}

\newcommand{\Vt}{\tilde{V}}

\newcommand{\pt}{\tilde{p}}

\newcommand{\st}{\tilde{s}}
\newcommand{\ut}{\tilde{u}}
\newcommand{\wt}{\tilde{w}}

\newcommand{\rhot}{\tilde{\rho}}

\newcommand{\Ubh}{\widehat{\overline{\CU}}}

\newcommand{\wht}{\tilde{\hat{w}}}
\newcommand{\pht}{\tilde{\hat{p}}}
\newcommand{\sht}{\tilde{\hat{s}}}
\newcommand{\rhoht}{\tilde{\hat{\rho}}}
\newcommand{\uht}{\tilde{\hat{u}}}

\renewcommand{\div}{\text{\rm div}\,}
\newcommand{\grad}{\nabla}

\newcommand{\dist}{\text{\rm dist}\,}
\newcommand{\DT}{\text{\rm D}}
\newcommand{\po}{\partial}
\newcommand{\dx}{\mathrm{d}\,}

\newcommand{\lbt}{\llbracket}
\newcommand{\rbt}{\rrbracket}

\theoremstyle{plain}
\newtheorem{theorem}{Theorem}[section]
\newtheorem{corollary}{Corollary}[section]

\newtheorem*{proba}{Problem A}
\newtheorem*{probb}{Problem B}
\newtheorem*{probc}{Problem MBV}

\newtheorem{lemma}{Lemma}[section]

\newtheorem{proposition}{Proposition}[section]

\theoremstyle{definition}
\newtheorem{definition}{Definition}[section]
\theoremstyle{remark}

\newtheorem{remark}{Remark}[section]
\numberwithin{equation}{section}

\begin{document}
\title{Local stability of direct Mach configurations}

\author{Jun Chen}
\address{Three Gorges Mathematical Research Center, China Three Gorges University,
No. 8 University Road, Yichang, Hubei 443002, China}
\email{chenjun4@ctgu.edu.cn}

\author{Xuemei Deng}
\address{Three Gorges Mathematical Research Center, China Three Gorges University,
No. 8 University Road, Yichang, Hubei 443002, China}
\email{dxuemei@ctgu.edu.cn}

\author{Xiaoguang You}
\address{School of Mathematical Sciences, Jiangxi Science and Technology Normal University,
No. 589 Xuefu Avenue, Nanchang, Jiangxi 330038, China}
\email{xgyou@jxstnu.edu.cn}

\begin{abstract}
We investigate the local stability of  direct Mach configurations
for the two-dimensional steady compressible Euler equations. Given a
piecewise constant direct Mach configuration with subsonic flow behind both the reflected shock and the Mach stem, we prove that under a transversality condition on the
shock polar loops and a suitable condition on the truncation segment, every
sufficiently small perturbation of the incoming flow gives rise to a
Mach configuration that is a small perturbation of the background
one, in the sense that the reflected shock, the Mach stem and the slip
line, together with the downstream subsonic flow, stay close to their
background counterparts. The triple point is fixed at the origin, and the domain is truncated at a finite distance. 
Euler--Lagrange  coordinates are employed to transform
the unknown contact discontinuity curve (slip line) into a fixed boundary and to
write the full Euler system as a first-order elliptic system for the flow
direction and pressure coupled to transport equations for the Bernoulli
and entropy quantities. The key idea in dealing with the contact
discontinuity is to solve a mixed boundary value problem, in divergence
form with discontinuous coefficients, for a single elliptic equation, so that continuity of the scalar unknown and its conormal flux allows
reconstruction of the flow direction and pressure with the required
contact discontinuity conditions.
\end{abstract}


\maketitle


\section{Introduction} \label{sec-setup}

The phenomenon of shock reflection upon impingement on a solid wall was
first reported by Ernst Mach \cite{Mach} in 1878. After a long period of
stagnation, the subject was systematically investigated by John von
Neumann in the early 1940s; he provided the first comprehensive
theoretical and experimental analysis, established the fundamental
frameworks of the two-shock theory for regular reflection and the
three-shock theory for Mach reflection, and applied the shock polar
analysis -- which goes back to Meyer \cite{Meyer} and Busemann
\cite{Busemann}; see also \cite{CF} -- to determine the transition
criteria between the different reflection patterns. We refer to
\cite{vonNeumann} and to the classical monograph of Courant and Friedrichs
\cite{CF} for a comprehensive account. A fundamental observation, due to
von Neumann \cite{vonNeumann} and Courant--Friedrichs \cite{CF}, is that a
structure consisting of three shocks alone cannot satisfy the
Rankine--Hugoniot conditions: an additional contact discontinuity must
emanate from the triple point. Depending on the angle of incidence, either
a regular reflection or a Mach reflection occurs. In the latter case, the
incident shock, the reflected shock and the Mach stem meet at a triple
point, and a slip line emanates from it; this triple-shock structure is
called a \emph{Mach configuration}. The stability of Mach configurations
is of fundamental importance in gas dynamics, since a configuration that is
not structurally stable can hardly be expected to be observed; here stability
is understood in the structural sense, for the steady boundary value problem,
and no time evolution is considered. The subject has attracted considerable
attention; see, for instance, \cite{ChenSX1,ChenSX2,Serre} and the references
therein. In this paper we study the local stability of
\emph{direct} Mach configurations, for which the angle between the Mach
stem and the slip line is smaller than a right angle.

The regular reflection pattern, which consists of the incident shock and
the reflected shock only, has a simpler structure, and it is by now
well understood, although the available results are mainly for the
potential flow equation and in the self-similar setting. Chen and Feldman proved
the global existence of solutions for large-angle wedges in \cite{CF10} and provided a systematic account of this theory in the
monograph \cite{CF18}. The optimal regularity $C^{1,1}$ of
the solution at the corner where the reflected shock and the sonic circle meet was
established by Bae, Chen and Feldman \cite{BCF};
the convexity of the reflected shock was shown by Chen, Feldman and Xiang
\cite{CFX}. The same self-similar set-up covers the reflection
of a supersonic flow by a solid ramp, that is, the Prandtl--Meyer
reflection, for which Prandtl \cite{Prandtl} conjectured that the weak
shock configuration is the physically admissible one. This problem was
studied by Bae, Chen and Feldman \cite{BCF1,BCF2} and by Elling and Liu
\cite{EllingLiu}.

Mach reflection, by contrast, is much less understood. It involves three
shocks and a slip line, and it is closely related to the shock problems
for steady flows past wedges: when a uniform supersonic flow impinges on a
solid wedge, the wave pattern is determined by the interaction of a shock
with a solid boundary, and the steady shock problems past wedges can be
formulated as free boundary problems, with the shock front as the unknown
boundary; see Chen, Chen and Feldman \cite{CCF1} and, for
three-dimensional wedges, Chen, Chen and Xiang \cite{CCX}. The stability of Mach configurations for the full Euler system
was initiated by Shuxing Chen, who treated the steady direct Mach
configuration \cite{ChenSX1}, the pseudo-stationary Mach configuration
\cite{ChenSX2}, configurations with subsonic flow behind the Mach stem
and supersonic flow behind the reflected shock \cite{ChenSX3}, and
stationary wall reflection \cite{ChenSX4}; we also refer to \cite{Serre} for the unsteady self-similar shock reflection
problem. For a given uniform supersonic incoming flow, and after the point at which the
four discontinuities meet has been fixed, Fang, Liu and Yuan \cite{FangLiuYuan}
proved that a flat Mach configuration is unique among piecewise $C^1$ flows
subject to suitable conditions on the downstream subsonic flow; in the present
paper the incoming flow itself is perturbed and need not be uniform.

Our analysis builds on the Euler--Lagrange formulation developed in
\cite{ChenSX1}. A complete treatment of the steady local problem requires
compatible data on the truncation boundary, estimates that accommodate
singular behavior at the corners, and recovery of the original
first-order equations and shock conditions from the scalar elliptic
problem. We address these issues within a single construction.

With the triple point fixed, we prove local structural stability of
direct Mach configurations with two subsonic downstream regions under the transversality condition
\eqref{TC} and the truncation condition \eqref{con-trunc}. To our knowledge,
this is the first rigorous local structural stability result for the
present truncated problem in a class allowing corner singularities and
preserving the family of nearby constant Mach configurations. The theorem
concerns the two-dimensional steady full Euler equations in a finite
neighborhood of the triple point. The pseudo-stationary problem \cite{ChenSX2},
the case with supersonic flow behind the reflected shock \cite{ChenSX3},
and stationary wall reflection \cite{ChenSX4}
involve different equations or geometric constraints.

The truncation condition leaves the constant value of $w+\tau^-p$ to be
determined by the shock data. It is therefore compatible with changes in
the downstream polar intersection; see Remark~\ref{rem-trunc}. In the
elliptic reduction, we retain the undifferentiated boundary data
\eqref{con-v1}--\eqref{con-v2}. The reconstruction formula
\eqref{eqn-v1int} then recovers both first-order equations, the original
shock conditions, and continuity across the contact discontinuity.

At the corners, we establish weighted H\"older estimates by explicit
barriers. These estimates allow singular derivatives at the vertices and
yield $C^\alpha$ continuity there, together with piecewise $C^{1,\alpha}$
regularity along the interior of the contact discontinuity. The weights
account for the different regularity at a corner and at an interior point
of the contact discontinuity \cite{ChenDeng}. Finally, the transversality
condition $\tau^+\ne\tau^-$ connects the shock polar geometry with the
elliptic boundary problem: it makes the transformation $T$ in
\eqref{def-vv} invertible and ensures that the boundary operator in
\eqref{con-g5} is oblique.

The feature that distinguishes a Mach configuration from the regular
reflection pattern above is the slip line, that is, a contact
discontinuity attached to the triple point, and the stability of contact
discontinuities has been studied in several different settings. Bae
\cite{Bae} treated the contact discontinuity for steady subsonic flows in
an infinite duct, and Bae and Park \cite{BP} studied the same problem in
infinitely long nozzles; Chen, Huang, Wang and Xiang \cite{CHWX} obtained
steady Euler flows with large vorticity and characteristic discontinuities
in arbitrarily long nozzles; see also \cite{CXZ,ADV1}, where vortex
lines attached to profiles are located by the implicit function theorem. In
the present paper we take a different route, based on
the framework of Chen and Deng \cite{ChenDeng} for elliptic equations with
discontinuous coefficients in domains with corners: the two regions
separated by the contact discontinuity are treated as a whole domain,
so that the conditions on the contact discontinuity enter the elliptic
problem as compatibility conditions and are satisfied automatically by its
solution.

Consider the following two-dimensional steady
compressible Euler equations:
\begin{align}
	& \div (\rho \uu)=0,\label{Euler1}\\
	&\div \left(\rho{\uu\otimes\uu}\right)
	+\nabla p=0,\label{Euler2}\\
	&\div \left(\rho\uu\left(E+\frac{p}{\rho}\right)\right)=0,
	\label{Euler3}
\end{align}
where $\nabla =(\po_{x_1}, \po_{x_2})^\top $ is the gradient operator, $\uu=(u_1,
u_2)^\top$ is the velocity field, $\rho$ is the density, $p$ is the pressure, and $E$ is the specific total energy. We assume the gas is polytropic with adiabatic exponent $\gc>1$, so the energy $E$ is given by:
$$
E=\frac{1}{2}|\uu|^2+\frac{p}{(\gc-1)\rho}.
$$
The sound speed of
the flow is given by
$$
c:=\sqrt{\gc p/\rho},
$$
and the Mach number is defined by
$$
M:=\frac{|\uu|}{c}.
$$
We say the flow is subsonic if the Mach number satisfies $M <1$ at every point of the flow, and supersonic if $M >1$ at every point of the flow.

{\bf Rankine--Hugoniot conditions on discontinuity curves.} Suppose that an open domain $\GO$ in $\R^2$ is divided by a $C^1$ curve $\CC$ into subdomains $\GO^+$ and $\GO^-$. Assume that $U=(\rho,u_1,u_2,p)^\top$ is a piecewise $C^1$ solution of Euler equations \eqref{Euler1}--\eqref{Euler3} in each domain $\GO^+$ and $\GO^-$, and $U$ is continuous up to the boundary $\CC$ in each subdomain. We denote the restriction of $U$ on $\GO^+\cup \CC$ by $U^+$ and on $\GO^-\cup \CC$ by $U^-$. If $U$ is a weak solution to \eqref{Euler1}--\eqref{Euler3} in the whole domain $\GO$, integration by parts will give rise to the following so-called Rankine--Hugoniot conditions on the curve $\CC$:
\begin{align}
	& \lbt  \rho \uu \rbt \cdot \nn =0,\label{RH1}\\
	&\left\lbt\rho{\uu\otimes\uu}+pI\right\rbt
	\nn  =\mathbf{0},\label{RH2}\\
	&\left\lbt \rho\uu \left(E+\frac{p}{\rho}\right)\right\rbt \cdot \nn=0,
	\label{RH3}
\end{align}
where $\nn$ is the unit normal vector on $\CC$, and the bracket $\lbt\ \rbt$ denotes the jump of the dependent variable from one subdomain to the other, i.e., for any smooth function $f: \R^4 \to \R$, $\lbt f(U)\rbt = f(U^+)- f(U^-)$.
Taking the scalar product of \eqref{RH2} with $\nn$ and with a unit tangential vector $\gtt$, respectively, gives
\begin{align}
	&  \lbt\rho (\uu \cdot \nn)^2 +p \rbt =0,\label{RH4}\\
	&\lbt \rho (\uu \cdot \nn) (\uu \cdot \gtt) \rbt  =0\label{RH5}.
\end{align}
When the flow travels across $\CC$, i.e., $\uu \cdot \nn \neq 0$ on $\CC$, and satisfies the entropy condition (the entropy, or equivalently, the pressure $p$, increases across $\CC$ ), the curve $\CC$ is called a shock; if the flow travels along both sides of $\CC$ so that $\uu \cdot \nn \equiv 0$ on $\CC$, $\CC$ is said to be a contact discontinuity or a characteristic discontinuity. In the latter case, condition \eqref{RH4} and the slip condition
\begin{align}
	& \uu^{\pm} \cdot \nn = 0\label{RH6}
\end{align}
imply
\begin{align}
	& p^+ = p^-\label{RH7}
\end{align}
along the contact discontinuity $\CC$. It is obvious that conditions \eqref{RH6} and \eqref{RH7} together ensure
the Rankine--Hugoniot conditions \eqref{RH1}--\eqref{RH3}. Denote the flow direction by
\begin{align}
	& w= \frac{u_2}{u_1}\label{def-w},
\end{align}
and the slip condition \eqref{RH6} gives rise to
\begin{align}
	&  w^+= w^- \label{con-w}
\end{align}
along $\CC$. Throughout the paper, $U=(\rho,u_1,u_2,p)^\top$
denotes the physical state, while $\CU=(\rho,u_1,w,p)^\top$ denotes
the same state in the variables used for the elliptic and transport
equations. These variables are related by $w=u_2/u_1$ and $u_2=u_1w$.
We use the same superscripts and subscripts for the corresponding states
$U$ and $\CU$.

\begin{figure}[t]
\begin{minipage}[t]{0.52\textwidth}
	\centering
	\includegraphics[width=\linewidth]{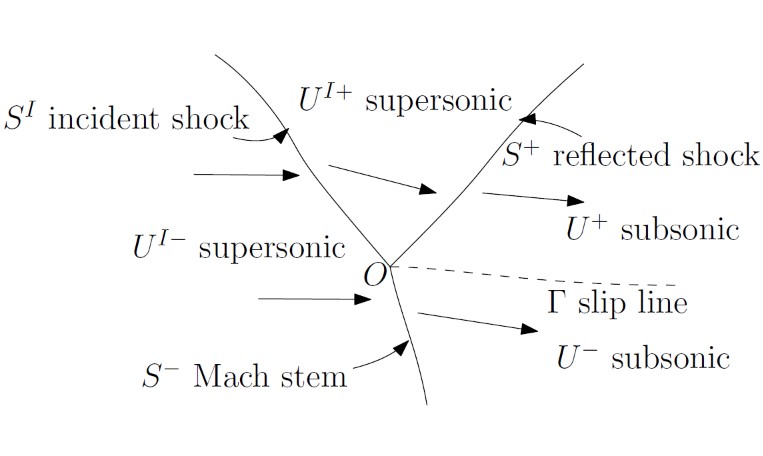}
	\captionof{figure}{Local structure of a direct Mach configuration.\label{fig:machconfig}}
\end{minipage}%
\hfill
\begin{minipage}[t]{0.48\textwidth}
	\centering
	\includegraphics[width=\linewidth]{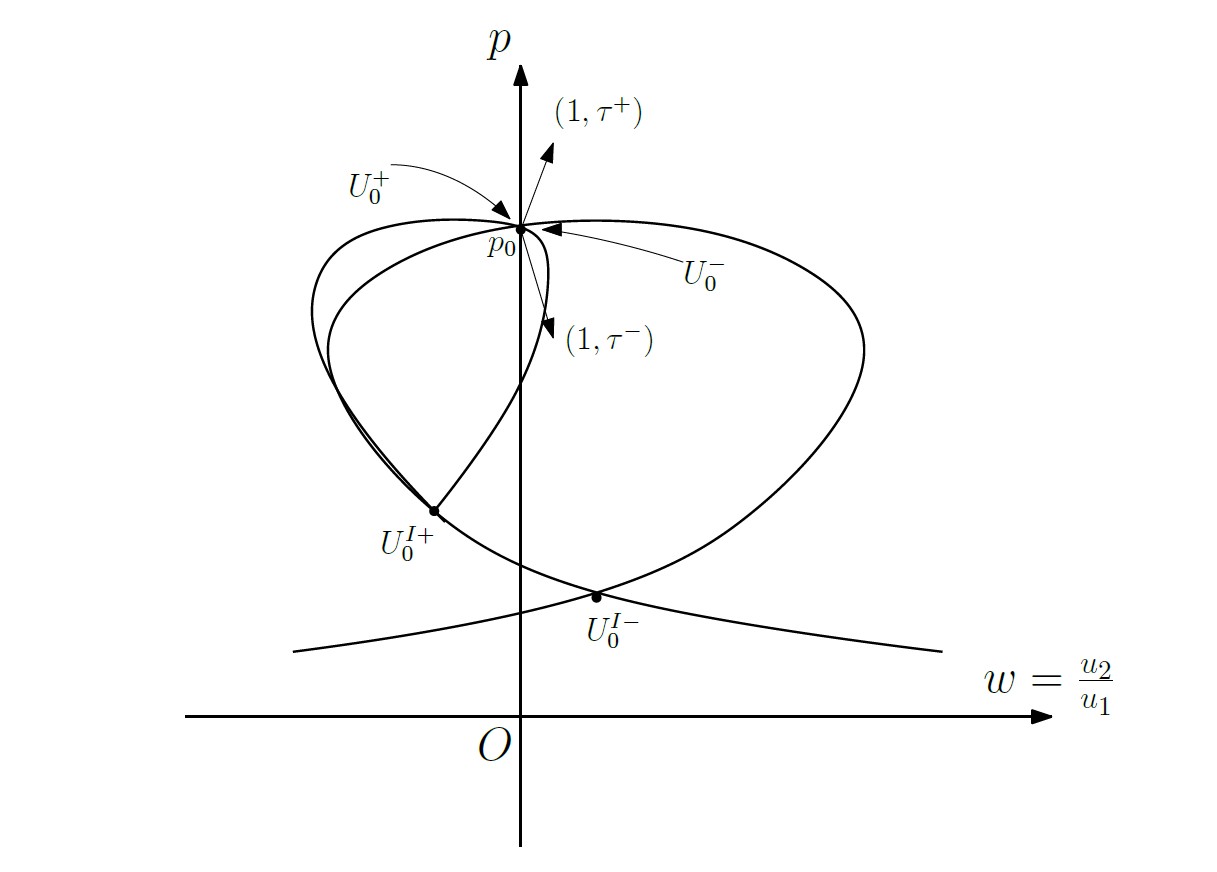}
	\captionof{figure}{Shock polar loops for a direct Mach configuration.\label{fig1}}
\end{minipage}
\end{figure}

A Mach reflection may occur when an incident shock $S^I$ impinges on a solid wall.
We are concerned with the stability of the local structure of the Mach configuration.
Figure \ref{fig:machconfig} illustrates the classical triple-shock structure of the Mach configuration:
$U^{I-}$ denotes the incoming flow ahead of the incident shock $S^I$; $U^{I+}$ is the incoming flow behind the shock $S^I$; $S^+$ is the reflected shock; $S^-$ is the Mach stem; and $\GC$ is the slip line, which is a contact discontinuity curve.
The downstream flows $U^+$ and $U^-$ are both subsonic and are separated by $\GC$.

Consider a piecewise constant Mach configuration with four constant states $U^{I-}_0, U^{I+}_0,U^-_0, U^+_0$ separated by straight shocks $S^I_0$, $S^+_0$, $S^-_0$ and the slip line $\Gamma_0$. Based on the Rankine--Hugoniot conditions \eqref{RH1}--\eqref{RH3} and the contact discontinuity conditions \eqref{con-w} and \eqref{RH7}, we may employ shock polar analysis to determine the downstream constant states $U^-_0, U^+_0$ when the states of the incoming flow $U^{I-}_0, U^{I+}_0$ are given. In fact, given $U^{I-}_0$, all possible states $U^{I+}_0$ connected to $U^{I-}_0$ through a shock form a one-parameter curve. Due to the conditions \eqref{con-w} and \eqref{RH7} across the slip line, it is convenient to use the variables $w$ and $p$ to draw the curve. Each polar loop is the locus, in the $w$-$p$ plane, of the states that can be connected to the state from which the loop is drawn by a single shock. In the $w$-$p$ plane, we choose a supersonic state $U^{I+}_0$ on the loop starting from $U^{I-}_0$ and draw another shock polar loop starting from $U^{I+}_0$. The intersection point of the two loops corresponds to the downstream states $U^-_0, U^+_0$, whose $p$ and $w$ values are equal, and hence satisfies conditions \eqref{con-w} and \eqref{RH7}. For our convenience, we rotate the coordinate system so that the downstream flow is parallel to the $x_1$-axis with positive orientation, or equivalently, $w^+_0=w^-_0 =0$ and $u_{10}^\pm>0$ in $\GO^\pm_0$ (see Figure \ref{fig1}). For sufficiently small perturbations, the downstream velocity retains a positive $x_1$-component. Together with positive upper and lower density bounds, this ensures that the change of variables in Section~\ref{sec-lagrange} is bi-Lipschitz.

 We also assume that both $U^-_0$ and $U^+_0$ are subsonic states:
\begin{align}
M_0^+ <1, \quad  M_0^- <1. \label{con-subsonic}
\end{align}

Let $(1,\tau^-)$ and $(1,\tau^+)$ be normal vectors to the loops from $U^{I-}_0$ to $U^{-}_0$ and from $U^{I+}_0$ to $U^{+}_0$, respectively, at the intersection point $(0,p_0)$. We assume that the two loops intersect transversally, equivalently, the following transversality condition holds:
\begin{align}
\tau^+\neq \tau^-. \tag{TC} \label{TC}
\end{align}

The prescribed incoming states are required to solve the steady Euler
system in their respective incoming regions. The downstream boundary
value problem uses only these states near the two outgoing shocks.
To interpret its solution as a complete Mach configuration, we impose
the additional compatibility with an incident shock stated below.

\begin{definition}\label{def-adm}
Let $U^{I\pm}_0$ be the two incoming constant states of the background
configuration. Let $U^{I\pm}$ be $C^{2,\alpha}$ supersonic states defined on
neighbourhoods of $\overline{\GO^{I\pm}_0}$. We say that $(U^{I-},U^{I+})$ is an
{\bf admissible incoming flow} if
\begin{enumerate}
\item[(i)] $U^{I\pm}$ solve Euler system \eqref{Euler1}--\eqref{Euler3} in their
 domains, with $u^{I\pm}_1>0$ there;
\item[(ii)] $U^{I-}$ and $U^{I+}$ are the two sides of a single incident shock:
 there is a $C^{2,\alpha}$ curve $S^I$ through the origin such that, in the two
 open regions into which $S^I$ divides a neighbourhood of $O$, the states
 $U^{I-}$ and $U^{I+}$ are defined and satisfy the Rankine--Hugoniot conditions
 \eqref{RH1}--\eqref{RH3} and the entropy condition across $S^I$;
\item[(iii)] the prescribed incident shock occupies the incoming sector
of the background configuration near the origin; the origin is fixed as
the triple point.

\end{enumerate}
\end{definition}

Mass conservation in (i) gives the incoming stream function and the
Lagrangian coordinate $y_2$. The energy and momentum equations then imply
that the incoming Bernoulli and entropy quantities are constant along
streamlines. On each incoming region they can therefore be written as
$\BB^{I\pm}(y_2)$ and $\CA^{I\pm}(y_2)$; mass conservation alone does not
give these transport laws.
 Condition (ii) is needed only to read the result
as a complete Mach configuration and will be used only in Corollary
\ref{cor-mach}; the construction below uses the incoming states only through
their traces near $S^{\pm}$, whose regularity is controlled by \eqref{con-UI}.

We perturb the incoming flow $(U^{I-}_0, U^{I+}_0)$ to $(U^{I-} , U^{I+})$ and the incident shock $S^I_0$ to $S^I$, and investigate if the Mach configuration is stable under small perturbations. In other words, we want to show that there exist a reflected shock $S^+$, a Mach stem $S^-$, a slip line $\GC$ and downstream states $U^+,U^-$, which are small perturbations of the given piecewise constant Mach configuration. The procedure of solving the problem shows that the values of $(U^{I-} , U^{I+})$ near the background shocks $S^+_0$ and $S^-_0$ (no information about $S^I$) will be enough to determine the downstream flow. Thus, we set up the problem as follows.

Let $s_0^{\pm}(x_1)= k^{\pm}x_1$ be the shock functions of $S_0^{\pm}$, where $k^+>0$, and $k^-<0$. Let $s^{\pm}$ (with $s^{\pm}(0)=0$) be perturbed shock functions for shocks $S^{\pm}$, $\gz$ (with $\gz(0)=0$) be the position for the contact discontinuity curve $\GC$. Truncate the domain at $x_1=l$, with $l>0$. Denote
 \begin{align*}
 S_0^{\pm} &= \{ \xx: x_2 = k^{\pm}x_1, 0<x_1 <l \},\\
  S^{\pm} &= \{ \xx: x_2 = s^{\pm}(x_1), 0<x_1 <l \},\\
  \II &= (0,l)\times \{0 \},\\
   \GC &= \{ \xx: x_2 = \gz(x_1), 0<x_1 <l \},\\
 \GO^{I+}_0 &=  \{ \xx: \tfrac{1}{2} k^+x_1<x_2 < 2k^+ x_1, 0<x_1 <l \},\\
 \GO^{I-}_0 &=  \{ \xx: 2k^- x_1< x_2 < \tfrac{1}{2} k^- x_1, 0<x_1 <l \},\\
 \GO^+ &=  \{ \xx: \gz(x_1)<x_2 < s^+(x_1), 0<x_1 <l \},\\
 \GO^- &=  \{ \xx:   s^-(x_1)<x_2<\gz(x_1), 0<x_1 <l \},\\
 \LL &= \{ \xx:  x_1=l, s^-(l) <x_2  < s^+(l) \}.
 \end{align*}
We prescribe the following condition on $\LL$:
\begin{align}
\tfrac{\po}{\po x_2}(w+\tau^- p ) = 0. \label{con-trunc}
\end{align}

\begin{remark}\label{rem-trunc}
Our exit condition preserves the nearby constant Mach configurations.
For instance, moving the incoming state $U_0^{I+}$ along the shock polar
issuing from $U_0^{I-}$ preserves the incident shock relation and moves the
downstream polar intersection to $(\wt_0,\pt_0)$; see
Figure~\ref{fig:conditionL}. Condition \eqref{con-trunc} is satisfied by
each such constant configuration. For the perturbed flow, it leaves the
value of $w+\tau^-p$ on $\LL$ to be determined by the boundary conditions.
By \eqref{con-v2}, this value equals its value at $P^-$, so no additional
scalar boundary value is prescribed.
\end{remark}

\noindent\begin{minipage}[t]{0.46\textwidth}
	\centering
\includegraphics[width=0.75\linewidth]{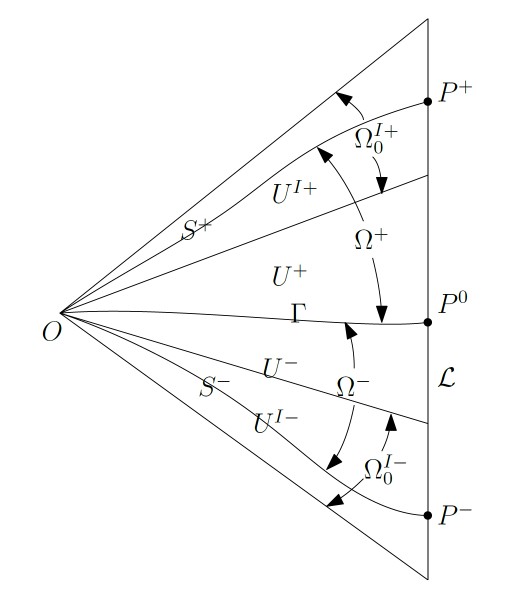}
 \captionof{figure}{Domain of Problem A.}\label{fig:domain}
\end{minipage}%
\hfill
\begin{minipage}[t]{0.54\textwidth}
	\centering
	\includegraphics[width=0.9\linewidth]{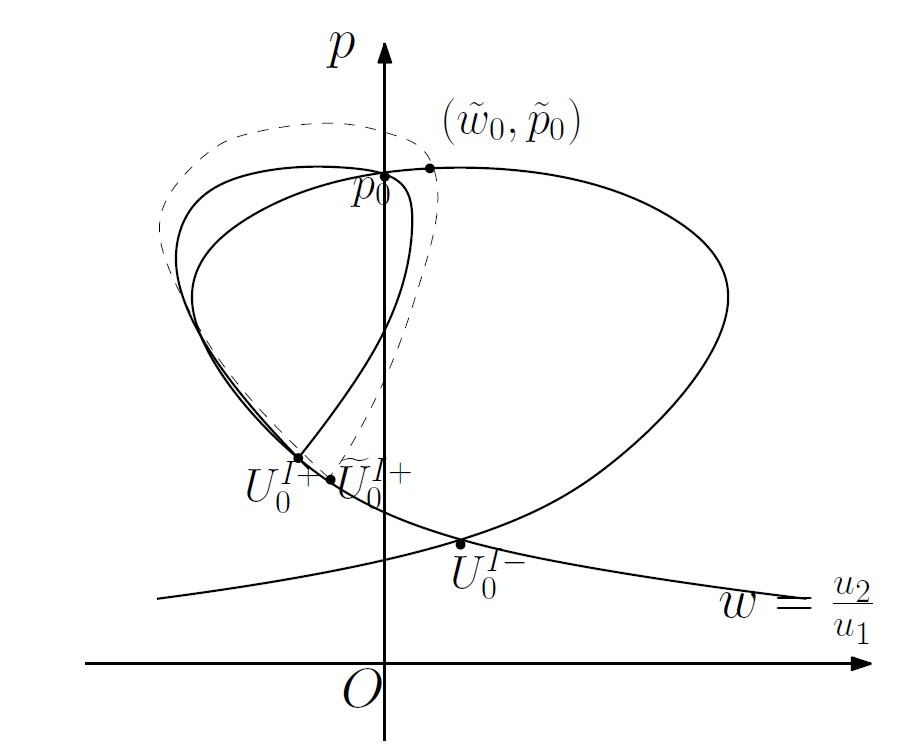}
	\captionof{figure}{Condition on $\LL$. \label{fig:conditionL}}
\end{minipage}

\begin{proba}
Suppose that a piecewise constant Mach configuration $U^{I-}_0$, $U^{I+}_0$, $U^-_0$, $U^+_0$ described above satisfies the transversality assumption \eqref{TC} and subsonicity condition \eqref{con-subsonic} for the downstream flow. Let $U^{I-},U^{I+}$ be $C^{2,\ga}$ supersonic solutions of the steady Euler equations on neighbourhoods of the respective incoming regions, with $u_1^{I\pm}>0$, and suppose that they are small perturbations of $U^{I-}_0,U^{I+}_0$ in $\GO^{I-}_0,\GO^{I+}_0$. Compatibility with an incident shock is imposed separately in Corollary~\ref{cor-mach}. 

 Find the reflected shock $S^+$ in $ \GO^{I+}_0$, the Mach stem $S^-$ in $ \GO^{I-}_0$, the contact discontinuity curve $\GC$, and the downstream solution $ U^+$, $U^-$ in $\GO^+,\GO^-$ respectively, so that the following are satisfied:

 \begin{enumerate}
 	\item $U^+$ and $U^-$ are solutions to Euler equations \eqref{Euler1}--\eqref{Euler3} in $ \GO^+$ and $ \GO^-$ respectively.
 	\item Shock conditions \eqref{RH1}--\eqref{RH3} hold on $S^+$ and $S^-$.
 	\item Contact discontinuity conditions \eqref{RH6}, \eqref{RH7} hold on $\GC$.
 \item Condition \eqref{con-trunc} is satisfied on $\LL$.
  \item Downstream solution $U^{\pm}$, the shocks $S^{\pm}$ and the contact discontinuity $\GC$ are all small perturbations of the background states $U^{\pm}_0$, $S_0^{\pm}$, and $\II$, respectively.
 \end{enumerate}
\end{proba}

Let $\ga \in (0,1)$, $\gb,\gb',\gbb \in \mathbb{R}$, $\gb \le \gb' $, $k$ be
a nonnegative integer, $\kk = (k_1, k_2)$ be an integer-valued
vector with $k_1, k_2 \ge 0$, $|\kk|=k_1 +k_2$, and $\DT^{\kk}=
\po_{x_1}^{k_1}\po_{x_2}^{k_2}$. Let $G$ be a bounded open set and let $\GG$ be a subset of $\po G$. Set \begin{align*}
d^{\GG}_{\xx} = \min( \dist(\xx,\GG),1), \qquad
d^{\GG}_{\xx, \xx'} = \min(d^{\GG}_{\xx}, d^{\GG}_{\xx'}).
\end{align*}

Suppose that $\GG\subset\po G$ is a continuous curve with endpoint set $\GG'=\{P_1,P_2\}$, and that $\PP\subset\po G\setminus\overline{\GG}$ is a finite set.
 For a function $f$ defined on $G$, we define a weighted H\"older norm with different weights at $\GG, \GG', \PP$. Suppose that $f$ is $C^{-\gb}$ up to the boundary portion $\GG$ and $C^{-\gb'}$ to $P_1,P_2$ (assume $\gb, \gb' <0$ for the moment). Since the regularity of $f$ at the endpoints should be no higher than the regularity up to the boundary curve, it is reasonable to assume that $\gb \le \gb' $. Let $\gbb$ be the weight index to $\PP$. The regularity property of $f$ can be described by the following norms.
\begin{align}
	\label{def-normsemi}
	&[ f ]_{k,0;G}^{(\gb,\gb';\GG)(\gbb;\PP)}
	:= \sup_{ \begin{subarray}{c}
				\xx\in G\\
			|\kk|=k
	\end{subarray}}
\begin{array}{l}
	((d^{\GG}_{\xx})^{\max(k+\gb,0)} (d^{\GG'}_{\xx})^{\max(\gb' + \min(k,-\gb),0)} \\
	\cdot (d^{\PP}_{\xx})^{\max(k+\gbb,0)}|\DT^\kk f(\xx)|),
\end{array}\\
	\label{def-norm} &[ f ]_{k,\ga;G}^{(\gb,\gb';\GG)(\gbb;\PP)}
	:= \sup_{ \begin{subarray}{c}
		 	\xx,\xx'\in G\\
		 	\xx\ne\xx'\\
		 	|\kk|=k
	\end{subarray}}
\begin{array}{l}
	((d^{\GG}_{\xx, \xx'})^{\max(k+\ga+\gb,0)} (d^{\GG'}_{\xx, \xx'})^{\max(\gb' + \min(k+\ga,-\gb),0)} \\
	\cdot (d^{\PP}_{\xx, \xx'})^{\max(k+\ga+\gbb,0)}\tfrac{|\DT^\kk f(\xx)-\DT^\kk f(\xx')|}{|\xx-\xx'|^\ga}),
\end{array}
	\\
	& \|f\|_{k,\ga;G}^{(\gb,\gb';\GG)(\gbb;\PP)}
	:= \sum_{i=0}^k [ f ]_{i,0;G}^{(\gb,\gb';\GG)(\gbb;\PP)}
	+ [ f ]_{k,\ga;G}^{(\gb,\gb';\GG)(\gbb;\PP)}.
	\label{def-normvector}
\end{align}
Denote
\begin{align}\label{def-Ckga}
C^{k,\ga}_{(\gb,\gb';\GG)(\gbb;\PP)}(G)	:= \{f: \|f\|_{k,\ga;G}^{(\gb,\gb';\GG)(\gbb;\PP)} <\infty\}.
\end{align}
For functions on $(0,l)$, set $d_t:=\min(t,l-t,1)$ and $d_{t,t'}:=\min(d_t,d_{t'})$. We use the endpoint-weighted norm
\begin{align*}
\|f\|_{k,\ga;(0,l)}^{(\gb;\{0,l\})}
&:=\sum_{j=0}^k\sup_{0<t<l}d_t^{\max(j+\gb,0)}|f^{(j)}(t)|\\
&\quad+\sup_{t\ne t'}d_{t,t'}^{\max(k+\ga+\gb,0)}
 \frac{|f^{(k)}(t)-f^{(k)}(t')|}{|t-t'|^\ga}.
\end{align*}
Let $P^+, P^0,P^-$ be the intersection points of $\LL$ with $S^+, \GC, S^-$ respectively. Define
$U^I$ in $\GO^I_0 := \GO^{I+}_0 \cup \GO^{I-}_0 $ by
\begin{align*}
U^I(\xx)= 	\left\{
\begin{array}{ll}
 U^{I+}(\xx), & \xx \in \GO^{I+}_0\\
U^{I-}(\xx), & \xx \in \GO^{I-}_0
\end{array}
\right. .
\end{align*}
All norms of the incoming perturbation on the disconnected set $\GO_0^I$
are understood separately on its two components:
\[
 \|U^I-U_0^I\|_{m,\ga;\GO_0^I}
 :=\sum_{\sigma=+,-}
 \|U^{I\sigma}-U_0^{I\sigma}\|_{m,\ga;\GO_0^{I\sigma}}.
\]
In particular, no H\"older quotient is taken between the two incoming
states across the triple point. Since $u_1$ is bounded away from zero,
the unweighted $C^{m,\ga}$ perturbation norms of $U^I$ and $\CU^I$
are equivalent near the background states, for $m=2,3$. Assume that
	\begin{align}
	\|U^I  - U^I_0\|_{2,\ga;\GO^I_0 } \le \ve. \label{con-UI}
\end{align}
Set
$$\PP:= \{P^+, P^-\}.$$
Define the norm for the downstream variables $\CU$ in $\GO:= \GO^+\cup \GO^-$ as follows:
\begin{align}
&\|f\|_{2,\ga;\GO}^{(-\ga;\PP)(\gb,\gb';\GC)}:=\|f\|_{2,\ga;\GO^+}^{(-\ga;\{P^+\})(\gb,\gb';\GC)}+\|f\|_{2,\ga;\GO^-}^{(-\ga;\{P^-\})(\gb,\gb';\GC)},\label{def-fGO}\\
&\|\CU\|^*_{\GO}:=\|(\rho,u_1)\|_{2,\ga;\GO}^{(-\ga;\PP)(-\ga,-\ga;\GC)}+\|(w,p)\|_{2,\ga;\GO}^{(-\ga;\PP)(-1-\ga,-\ga;\GC)}.\label{def-UGO}
\end{align}

\begin{remark}
The regularities of $u_1,\rho$ and $w,p$ are different: on each side, $u_1,\rho$ are $C^\ga$ up to the contact discontinuity, while $w,p$ are $C^{1,\ga}$ up to compact subsegments of $\GC$ away from its endpoints. The difference comes from the different roles of these variables in the coupled elliptic and transport equations. In fact, $w,p$ are obtained by solving a first-order elliptic system. Hence, $w,p$ are piecewise $C^{1,\ga}$ up to the relative interior of $\GC$, whereas only $C^\ga$ regularity is asserted at the corners $\O, P^0$. On the other hand, $u_1,\rho$ are recovered from $w,p$ and the Bernoulli and entropy quantities, which are transported along streamlines. The $C^\ga$ regularity available for these transported quantities yields the stated regularity of $u_1,\rho$.
Since $u_2=u_1w$, both velocity components have the same weaker weighted
regularity as $\rho$. The norm of $\CU$ also retains the stronger control
of the flow direction $w$ and the pressure $p$ needed in the estimates
and the uniqueness argument.
\end{remark}

Denote the set of data for the background state and the domain by
\begin{align*}
\Xi := \{ U^{I-}_0, U^{I+}_0, U^-_0, U^+_0, l\}.
\end{align*}
In estimates containing $s^\pm$, the norm denotes the sum of the two one-sided norms. The main results can now be stated as follows.

\begin{theorem} \label{thm1}
Assume the background configuration satisfies \eqref{TC} and
\eqref{con-subsonic}, with $k^-<0<k^+$ as above. There exists
$\ga_0\in(0,1)$, depending on $\Xi$, such that, for every
$\ga\in(0,\ga_0)$, there is $\ve_0>0$, depending on $\Xi$ and $\ga$,
with the following property. If the prescribed incoming states solve the
steady Euler equations and satisfy \eqref{con-UI} with
$0<\ve\le\ve_0$, then {\bf Problem A} has a solution satisfying
\begin{align}
&\|\CU-\CU_0\|^*_{\GO} + \|s^{\pm}-s_0^{\pm}\|_{3,\ga;(0,l)}^{(-\ga-1;\{0,l\})}+ \|\gz\|_{2,\ga;(0,l)}^{(-\ga-1;\{0,l\})}\nonumber\\
&\hspace{30mm}\le C\|U^I - U^I_0\|_{2,\ga;\GO^I_0}, \label{est-U1}
\end{align}
where $C$ is a positive constant depending on $\Xi$ and $\ga$.

If we assume that
\begin{align}
\|U^I  - U^I_0\|_{3,\ga;\GO^I_0 } \le \ve, \label{con-unique-UI}
\end{align}
and the solution satisfies
\begin{align}
	\|\CU-\CU_0\|^*_{\GO} + \|s^{\pm}-s_0^{\pm}\|_{3,\ga;(0,l)}^{(-\ga-1;\{0,l\})} \le C \ve, \label{est-unique-U}
\end{align}
then the solution to {\bf Problem A} satisfying \eqref{est-unique-U} is unique, after reducing $\ve_0$ if necessary according to the fixed constant $C$ in that bound.
\end{theorem}

The main theorem states that, under the transversality condition
\eqref{TC} and the subsonicity condition \eqref{con-subsonic}, the direct
Mach configuration is locally stable: sufficiently small perturbations of
the incoming flow give rise to a nearby Mach configuration, which is
unique in the class of solutions satisfying the a priori estimate
\eqref{est-U1} (for incoming flows that are small in the stronger norm
\eqref{con-unique-UI}). The
proof is based on the following strategy. First, introducing the stream
function and passing to Euler--Lagrange coordinates, the streamlines and
the slip line are straightened; in particular, the contact discontinuity,
which is unknown in the original variables, becomes a fixed straight
boundary segment. The shocks are subsequently straightened by the
additional transformation in Section~\ref{sec-linear}.
 In these coordinates, the Euler system is
reduced to the Bernoulli law, the conservation of entropy along
streamlines, and a first-order elliptic system for the flow direction $w$
and the pressure $p$; see Section~\ref{sec-lagrange}. The
hyperbolic--elliptic coupling is then resolved by eliminating the
hyperbolic variables through the Rankine--Hugoniot conditions, which leads
to a mixed boundary value problem, for a single elliptic equation, with a
Dirichlet condition on the Mach stem, a Dirichlet condition on the
truncation segment $\LL$, and an oblique derivative condition on the
reflected shock; see Section~\ref{sec-linear}. In Section~\ref{sec-mbvp},
the solvability of this mixed boundary value problem is established by the
Perron method. The a priori estimates rely on weighted H\"older spaces
adapted to the corners and to the contact discontinuity; the estimates
across the contact discontinuity, which is attached to the corners, are
obtained from the results of Chen and Deng \cite{ChenDeng} for elliptic
equations with discontinuous coefficients in domains with corners.
Finally, in Section~\ref{sec-mach}, a Schauder fixed point argument yields
the existence stated in Theorem~\ref{thm1}, while uniqueness requires
higher regularity of the incoming flow.

\begin{remark}[Scope of the result]\label{rem-scope}
Theorem \ref{thm1} is a statement of local structural stability for the steady
Euler boundary value problem: the triple point is fixed at the origin, the domain
is truncated at $x_1=l$, and the flow satisfies condition \eqref{con-trunc},
whose constant is determined by the Mach stem endpoint as explained in Remark
\ref{rem-trunc}. No time evolution is studied, and the result is not a statement
about the dynamical stability of the reflection pattern.
\end{remark}

\begin{corollary}\label{cor-mach}
Let $(U^{I-},U^{I+})$ be an admissible incoming flow in the sense of Definition
\ref{def-adm} satisfying \eqref{con-UI}, and let $(U,s^{\pm},\gz)$ be the solution
of {\bf Problem A} given by Theorem \ref{thm1}. Then, together with the incident
shock $S^I$, the reflected shock $S^+$, the Mach stem $S^-$, the slip line $\GC$
and the downstream states $U^{\pm}$ form a Mach configuration near the origin, with that point as its
triple point. The configuration is a small perturbation of the background
one, and the estimates~\eqref{est-U1} hold.
\end{corollary}

The rest of the paper is organized as follows. In
Section~\ref{sec-lagrange}, we reformulate Problem~A in Lagrangian
coordinates and derive the reduced elliptic system. In
Section~\ref{sec-linear}, we linearize the problem and derive the mixed
boundary value problem. In Section~\ref{sec-mbvp}, we establish the
existence and uniqueness of solutions to the mixed boundary value problem
by the Perron method. In Section~\ref{sec-mach}, we complete the proof of
the main theorem by a fixed point argument.

\section{Reformulation of {\bf Problem A} in Lagrangian coordinates} \label{sec-lagrange}

Let $U$ be a piecewise smooth solution of the steady Euler system in a simply connected domain $G$, with $\rho u_1$ bounded away from zero. Mass conservation,
including the mass Rankine--Hugoniot condition at any internal
discontinuity, gives a continuous stream function $\psi$, determined by
\begin{align}
	&\psi_{x_1}= -\rho u_2, \quad \psi_{x_2} = \rho u_1,\label{stream}
\end{align}
with	$\psi(\O) =0$.

Define the following coordinate
transformation of Euler--Lagrange type
\begin{equation}\label{def-coord}
	\left\{
	\begin{array}{lll}
		y_1=x_1,\\
		y_2= \psi(x_1, x_2).
	\end{array}
	\right.
\end{equation}
Since $\psi$ is constant along the streamlines, the transformation $(x_1,x_2)\mapsto (x_1,\psi(x_1,x_2))$ straightens the streamlines; in particular, the slip line $\GC$, along which the velocity is tangential, is itself a streamline and is therefore mapped onto the fixed segment $\II=\{y_2=0\}$. Thus the contact discontinuity, which is an unknown free boundary in the $\xx$-coordinates, becomes a fixed boundary in the $\yy$-coordinates. For the domains considered here, bounded by the two Lipschitz shock graphs and the truncation segment, this transformation is globally invertible and bi-Lipschitz, provided that $0<c\le\rho u_1\le C$ and $|\rho u_2|\le C$ there; recall that this will be ensured by the small perturbation framework.

In the new coordinates $\yy=(y_1,y_2)$, we write $U(\yy)=U(\xx(\yy))$ and $\CU(\yy)=\CU(\xx(\yy))$, retaining the distinction between the two state vectors. The Euler equations \eqref{Euler1}--\eqref{Euler3} become
\begin{align}\label{eqn-euler1}
	&\Big(\frac{1}{\rho u_1} \Big)_{y_1}
	-\Big(\frac{u_2}{u_1} \Big)_{y_2}= 0,\\[1mm]
	&\Big(u_1 + \frac{p}{\rho u_1}\Big)_{y_1}
	-\Big(\frac{p u_2}{u_1}\Big)_{y_2}= 0, \label{eqn-euler2}\\[1mm]
	&(u_2)_{y_1} + p_{y_2}= 0,\label{eqn-euler3} \\[1mm]
	& \Big(\frac{1}{2}|\uu|^2 +\frac{\gc p}{(\gc-1)\rho}\Big)_{y_1}=0.
	\label{eqn-euler4}
\end{align}

Rankine--Hugoniot conditions on $S^{\pm}$ are the following:
\begin{align}\label{con-RH1}
	(s^{\pm})'(y_1)\left[\frac{1}{\rho u_1}\right]&=-\left[\frac{u_2}{u_1} \right] ,\\
	(s^{\pm})' (y_1)	\left[ u_1 + \frac{p}{\rho u_1}\right]&=-\left[\frac{p u_2}{u_1}\right]
	,
	\label{con-RH2}\\
	(s^{\pm})' (y_1)	[\,u_2 \,]&= [\,p\,] ,
	\label{con-RH3}\\
	\left[ \frac{1}{2}|\uu|^2 + \frac{\gc p}{(\gc-1)\rho}\right]&= 0.
	\label{con-RH4}
\end{align}
For simplicity of notation, we retain the same notations for $ \GO^{I-}_0 $, $\GO^{I+}_0 $, $S^+$, $S^-$, etc., after the coordinate transformation. $ \GO^+$, $ \GO^-$ are domains bounded by $S^+$, $S^-$, $\LL$ and $\II$. Hereafter, these quantities are in $\yy$-coordinates and the slopes $k^+,k^-$ are also different from those in Section \ref{sec-setup}.

Condition \eqref{con-trunc} should be rewritten correspondingly:
\begin{align}
	\tfrac{\po}{\po y_2}(w+\tau^- p ) = 0. \label{con-truncy}
\end{align}

For the prescribed incoming flow, the change of variables is constructed
independently in each incoming region from
$\psi^{I\pm}(O)=0$. The background outgoing shock is transverse to
the incoming streamlines, since
\[
 \frac{\mathrm d}{\mathrm dt}
 \psi^{I\pm}_0(t,s_{x,0}^\pm(t))
 =\rho^{I\pm}_0u^{I\pm}_{10}
   ((s_{x,0}^\pm)'-w^{I\pm}_0)\ne0,
\]
where $s_{x,0}^\pm$ denotes the shock in physical coordinates.
This transversality persists under small perturbations. Thus the
incoming Bernoulli function can be parametrized by $y_2$ on a
neighbourhood of every stream level meeting the possible shocks;
the steady energy equation makes this definition independent of the
point chosen on that streamline. The incoming coordinate maps and
their inverses have uniform $C^{3,\ga}$ bounds on these fixed
neighbourhoods. Consequently the transformed incoming perturbations
satisfy \eqref{con-UI}, up to a constant depending only on $\Xi$
and $\ga$; the same assertion holds for \eqref{con-unique-UI}.
We retain the notation $\ve$ after this fixed rescaling.

Then {\bf Problem A} in Section \ref{sec-setup} is restated in
$\yy$-coordinates as follows.
\begin{probb}
	Given the same Mach configuration $U^{I-}_0$, $U^{I+}_0$, $U^-_0$, $U^+_0$ as in {\bf Problem A}. Let $U^{I-},U^{I+}$ be the Lagrangian transforms of the prescribed incoming Euler flows in Problem A, with the smallness stated above.

	Find the reflected shock $S^+$ in $\GO^{I+}_0 $, the Mach stem $S^-$ in $ \GO^{I-}_0$, and the downstream solution $ U^+$, $U^-$, so that the following are satisfied:

	\begin{enumerate}
		\item $U^+$ and $U^-$ are solutions to Euler equations \eqref{eqn-euler1}--\eqref{eqn-euler4} in $ \GO^+$ and $ \GO^-$ respectively.
		\item Shock conditions \eqref{con-RH1}--\eqref{con-RH4} hold on $S^+$ and $S^-$.
		\item Conditions \eqref{RH7} and \eqref{con-w} hold on $\II$.
		\item Condition \eqref{con-truncy} is satisfied on $\LL$.
		\item The downstream states $U^\pm$ and shocks $S^\pm$ are small perturbations of $U_0^\pm$ and $S_0^\pm$; the contact discontinuity is the fixed segment $\II$ in these coordinates.
	\end{enumerate}
\end{probb}

Theorem \ref{thm1} is restated as follows.
\begin{theorem} \label{thm2}
	Under the background assumptions of Theorem~\ref{thm1}, there exists
$\ga_0\in(0,1)$ such that, for every $0<\ga<\ga_0$, there is
$\ve_0>0$ for which {\bf Problem B} has a solution whenever
\eqref{con-UI} holds with $0<\ve\le\ve_0$. Here the incoming states
are the Lagrangian transforms of the prescribed steady Euler flows.
The solution satisfies
	\begin{align}
		\|\CU-\CU_0\|^*_{\GO} +\|s^{\pm}-s_0^{\pm}\|_{3,\ga;(0,l)}^{(-\ga-1;\{0,l\})}\le C\|U^I  - U^I_0\|_{2,\ga;\GO^I_0 }, \label{est-U}
	\end{align}
	where $C$ is a positive constant depending on $\Xi$ and $\ga$.

If we further assume that \eqref{con-unique-UI}
 and \eqref{est-unique-U} hold, then the solution to {\bf Problem B} satisfying \eqref{est-unique-U} is unique.
\end{theorem}

Rewrite \eqref{eqn-euler1}--\eqref{eqn-euler4} in the non-divergence form
for the physical state $U=(\rho,u_1,u_2,p)^\top$:
\begin{equation}
	A_1(U) U_{y_1} + A_2(U) U_{y_2} =0, \label{euler-nondiv}
\end{equation}
where
\begin{align*}
	A_1(U) &=
	\begin{bmatrix}
		-\frac{1}{\rho^2 u_1}&-\frac{1}{\rho u_1^2} & 0 &0\\[1.5mm]
		-  \frac{ p}{ \rho^2 u_1}	&1 -  \frac{ p}{ \rho u_1^2} & 0 &  \frac{1}{\rho u_1}   \\[1.5mm]
		0& 	0  & 1& 0\\[1mm]
		- \frac{\gc p}{(\gc - 1) \rho^2}&	u_1 & u_2 & \frac{\gc}{(\gc - 1) \rho}
	\end{bmatrix},\\[1mm]
	A_2(U) &=
	\begin{bmatrix}
		0&	\frac{u_2}{ u_1^2} & -\frac{1}{u_1} &0 \\[1.5mm]
		0&	\frac{ pu_2}{ u_1^2} & - \frac{p}{u_1} & - \frac{u_2}{ u_1}   \\[1.5mm]
		0&	0  & 0& 1 \\[1mm]
		0& 0 &0 & 0
	\end{bmatrix}.
\end{align*}

Solving $\det (\gl A_1 -A_2)=0$ for $\gl$ gives four eigenvalues:
\begin{eqnarray*}
	&& \gl_1 =\gl_2 =0,\\
	&& \gl_{3,4} \equiv \gl_{\pm} =-\frac{c\rho}{c^2-u_1^2}\big(c u_2 \mp u_1 \sqrt{c^2 -q^2} i\big),
\end{eqnarray*}
where $q= \sqrt{u_1^2 + u_2^2} < c$ in the subsonic region.
The corresponding left-eigenvectors are
\begin{eqnarray*}
	l_1 & =& (0,0,0,1),\\
	l_2 & =&(-p u_1, u_1, u_2, -1),\\
	l_{3,4} &=&((\frac{\gc p^2}{(\gc -1) \rho u_1} - \frac{p u_1}{ \gc -1 })\gl_{3,4} + \frac{\gc p^2 u_2}{(\gc - 1) u_1},\\
	&&\,\,-(u_1+\frac{\gc p}{(\gc -1) \rho u_1})\gl_{3,4}-\frac{\gc p u_2}{(\gc -1)u_1}, \frac{\gc p}{\gc -1} - u_2 \gl_{3,4},  \gl_{3,4}).
\end{eqnarray*}

Then left-multiplying the Euler system \eqref{euler-nondiv} by $l_i$, $i=1, \cdots, 4$, leads to
Bernoulli's law:
	\begin{equation}\label{eqn-bernoulli}
		\frac{1}{2}q^2 + \frac{\gc p}{(\gc-1)\rho} = \BB(y_2),
	\end{equation}
conservation of entropy along smooth streamlines:
	\begin{equation} \label{eqn-entropy}
		\frac{p}{\rho^\gc} =\CA(y_2),
	\end{equation}
and the first-order elliptic system for $(w,p)$:
	\begin{align}
		&D_R w + \mu D_I p  =  0,   \label{eqn-wp1}\\
		& D_I w - \mu  D_R p  = 0, \label{eqn-wp2}
	\end{align}
	where
	\begin{align}
	&  D_R = \po_{y_1} + \gl_R \po_{y_2}, \quad  D_I = \gl_I\po_{y_2},
	  \quad \mu = \frac{\sqrt{c^2-q^2}}{c\rho u_1^2},\\
	  &	\gl_R =-\frac{c^2 \rho u_2}{c^2 - u_1^2 },  \quad \gl_I =  \frac{c\rho u_1 \sqrt{c^2 -q^2}}{c^2 - u_1^2}.\label{exp-drlambda}
	\end{align}

In a sufficiently small neighbourhood of the background states,
$p,\CA$ and $2\BB-2\gc p/((\gc-1)\rho)$ remain positive.
Equations \eqref{eqn-bernoulli} and \eqref{eqn-entropy} therefore
determine $\rho$, $u_1>0$ and $u_2$ as follows:
\begin{align}\label{eqn-u1}
	\rho= \left(\frac{p}{\CA}\right)^{1/\gc}, &&	u_1 = \frac{ \sqrt{2\BB - \frac{2\gc p }{ (\gc -1 )\rho }}}{\sqrt{1+ w^2}}, && u_2= wu_1.
\end{align}
Therefore, \eqref{eqn-euler1}--\eqref{eqn-euler4} are reduced to the elliptic system \eqref{eqn-wp1}--\eqref{eqn-wp2} for $(w,p)$, once the quantities $\CA,\BB$ are known. We point out that the Bernoulli quantity $\BB$ and the entropy quantity $\CA$ are functions of $y_2$ alone; the Bernoulli quantity is inherited from the incoming flow across the shocks, whereas the entropy is determined by the downstream shock traces. Both are then transported along downstream streamlines and couple the elliptic and hyperbolic parts of the flow.

\section{Linearization} \label{sec-linear}
Let $\DD$ be the open domain bounded by $S_0^+$, $S_0^-$ and $\LL$, and let $\DD^+$ and $\DD^-$ be its two subdomains separated by $\II$.

We still denote by $P^+, P^0,P^-$ the intersection points of $\LL$ with $S_0^+, \II, S_0^-$, respectively.
In this section $\LL$ denotes its part in $\partial\DD$, namely
$\{l\}\times[k^-l,k^+l]$. We construct a piecewise smooth,
bi-Lipschitz transformation $\TT^{\ess}:\GO\to\DD$. Write
$\sh^\pm(t):=s^\pm(t)-k^\pm t$, with $\sh^\pm(0)=0$.
Given $\zz\in\DD$, define
$$\yy= (\TT^{\ess})^{-1} \zz = (z_1, s(\zz)),$$
where $s$ is defined piecewise in $\DD^{\pm}$ as follows. 

\medskip
\noindent{\bf The extension of the shock functions $s^{\pm}(z_1)$ to $s(\zz)$.}
For $\zz\in\DD^+$, put $d_+(\zz):=k^+z_1-z_2$ and set
\[
 q_+:=(\sh^+)'(0),\qquad h_+(t):=\sh^+(t)-q_+t.
\]
The endpoint derivative exists because
$|(\sh^+)''(t)|\le C\|\sh^+\|_{3,\ga}^{(-1-\ga)}t^{\ga-1}$
near zero. In particular, $h_+(0)=h_+'(0)=0$.
Choose a nonnegative smooth function $\xi$ with support contained in
$[-1,1]$ and satisfying
\[
 \int_{\R}\xi(t)\,\mathrm dt=1.
\]
Let $\eta$ be an even smooth cutoff, nonincreasing on $[0,\infty)$,
such that $0\le\eta\le1$,
$\eta=1$ on $[-1,1]$ and $\eta=0$ outside $(-2,2)$.
For a parameter $c_e>0$, define
\begin{align}
 A_+(\zz,t)&:=z_1+c_ed_+(\zz)
 \left(t+\eta(\tfrac{z_1}{c_e})
             -\eta(\tfrac{z_1-l}{c_e})\right),\nonumber\\
 \sh^*(\zz)&:=\int_{\R}h_+(A_+(\zz,t))
          \frac1{c_e}\xi(\tfrac{t}{c_e})\,\mathrm dt,
          \label{exp-extend}\\
 \sh(\zz)&:=\frac{q_+}{k^+}z_2+
 \eta\left(4(\tfrac{z_2}{k^+z_1}-1)\right)\sh^*(\zz),
 \qquad s(\zz):=z_2+\sh(\zz).\label{def-s}
\end{align}
The parameter $c_e$, which is distinct from the background sound speed
$c_0$, is chosen to satisfy
\begin{equation}\label{con-c0}
 \begin{gathered}
 0<c_e\le\min(\tfrac14,\tfrac l8),\qquad
 c_e\max(k^+,-k^-)\le\tfrac12,\\
 c_e\max(k^+,-k^-)l\le1.
 \end{gathered}
\end{equation}

\begin{remark}\label{rem-extend}
These conditions ensure that $A_+(\zz,t)\in(0,l)$ for
$\zz\in\DD^+$ and $|t|\le c_e$. Indeed,
\begin{align*}
 A_+(\zz,t)&\ge z_1\{1-c_ek^+(1+c_e)\}>0,\\
 A_+(\zz,t)&\le4c_e\le l/2 &&(z_1\le2c_e),\\
 A_+(\zz,t)&\le z_1+c_e^2d_+(\zz)
                 \le z_1+c_e<l &&(2c_e<z_1<l-c_e),\\
 A_+(\zz,t)&\le z_1<l &&(l-c_e\le z_1<l).
\end{align*}
The first-order term $q_+z_2/k^+$ must be extended separately.
Multiplying $q_+z_1$ by the angular cutoff would create second
derivatives of order $|\zz|^{-1}$ and would violate the weight
$-1-\ga$ in \eqref{est-extend}.
\end{remark}

The extension has the exact boundary values
$s(z_1,k^+z_1)=s^+(z_1)$ and $s(z_1,0)=0$, and satisfies
\begin{align}
 \|\sh\|_{3,\ga;\DD^+}^{(-1-\ga;\{\O,P^+\})}
 \le C\|\sh^+\|_{3,\ga;(0,l)}^{(-1-\ga;\{0,l\})}.
 \label{est-extend}
\end{align}
To verify the estimate near $\O$, subtract the first-order term as
above. For sufficiently small $r>0$, use the change of variables
$\zz=r\zz'$ on $\DD^+\cap\{r<|\zz|<2r\}$.
Differentiating \eqref{exp-extend} shows that the $C^{3,\ga}$ norm
of the rescaled remainder is bounded by $Cr^{1+\ga}$ times
$\|\sh^+\|_{3,\ga;(0,l)}^{(-1-\ga;\{0,l\})}$.
Each derivative of the angular cutoff introduces a factor
$r^{-1}$, so the product rule gives the weights in \eqref{est-extend}.
Near $P^+$ the cutoff is identically one. Subtracting the affine
Taylor polynomial of $h_+$ at $l$ gives the same estimate with
$r=|\zz-P^+|$; the averaging in \eqref{exp-extend} extends that
affine polynomial to an affine function near $P^+$.
Away from these endpoints, ordinary differentiated convolution
estimates apply. These local estimates, and the usual comparison
of pairs at distances smaller or larger than the local radius,
give \eqref{est-extend}, including its H\"older seminorm.

For $\zz\in\DD^-$ use the same construction with
$d_-(\zz):=|k^-z_1-z_2|$, $q_-:=(\sh^-)'(0)$ and
$h_-(t):=\sh^-(t)-q_-t$. In a conical neighbourhood of $\II$,
\[
 s(\zz)=(1+q_\pm/k^\pm)z_2\quad\hbox{in }\DD^\pm.
\]
Thus $s$ is continuous across $\II$, although $s_{z_2}$ may jump
there. For small input norms, $s_{z_2}\ge1/2$ on both sides, so
the extension defines the required bi-Lipschitz transformation.

Denote
\begin{align*}
& \CU^{s}(\zz) := \CU((\TT^\ess)^{-1}(\zz)),
 \quad f^s :=f(\CU^s),
\end{align*}
where $f$ is any function of $\CU$, such as $\gl_I, \gl_R$ in \eqref{exp-drlambda}.

  Hence, equations \eqref{eqn-wp1}, \eqref{eqn-wp2} become the following equations in $\zz$-coordinates:
\begin{align}
	D^s_R w^s + \mu ^s D^s_I p^s& =   0,   \label{eqn-wp-z1}\\
	 D^s_I w^s - \mu ^s D^s_R p^s & =  0,\label{eqn-wp-z2}
\end{align}
where
\begin{align*}
D^s_R&= \partial_{z_1} +   \tfrac{\gl_R^s -s_{z_1}}{s_{z_2}} \partial_{z_2} ,&
D^s_I &= \tfrac{\gl_I^s}{s_{z_2}}  \partial_{z_2}.
\end{align*}
Rewrite \eqref{eqn-wp-z1} and \eqref{eqn-wp-z2} so that the coefficients of the derivative terms are piecewise constant:
\begin{align}
	\partial_{z_1} w^s + a\, \partial_{z_2} p^s& =   f_1,   \label{eqn-w-z1}\\
	\partial_{z_1} p^s - b \, \partial_{z_2} w^s & =   f_2,\label{eqn-p-z1}
\end{align}
where
\begin{align}
&a  =  \mu ^*\gl^*_I= \frac{1}{u_{10}},
\quad b  =  \frac{\gl^*_I}{\mu^*}= \frac{c_0^2\rho_0^2u_{10}^3}{c_0^2-u_{10}^2},\\
&f_1(s,\CU^s)  = ( D^*_R -D^s_R) w^s +(\mu ^* D^*_I-\mu ^s D^s_I)p^s, \label{def-f1}\\
&f_2(s,\CU^s)  = (D^*_R- D^s_R)p^s +\left( \frac{1}{\mu^s}D^s_I -\frac{1}{\mu^*}D^*_I\right) w^s,\label{def-f2}
\end{align}
and the superscript $*$ indicates evaluation at the background solution $\CU^s=\CU_0$ and $s=s^*(\zz)\equiv z_2$.
For the linearization, let $\Ub$ denote the reduced triple $(w,p,\rho)$.
Set
\begin{align*}
&\overline{\CU}^s= (w^s, p^s, \rho^s),&& \Ubh^s= (\wh^s, \ph^s, \rhoh^s) = \Ub^s - \Ub_0, \\
&U^{I\pm,s}(t,k^\pm t)=U^{I\pm}(t,s^\pm(t)),
&& \BB_s(\zz)=\BB^{I\pm}(s(\zz)),\quad \zz\in\DD^\pm.
\end{align*}
The incoming states $U^{I\pm,s}$ are needed only as traces on
$S_0^\pm$ (and in neighbourhoods of these traces). They are not
defined by composing $U^{I\pm}$ throughout the downstream domain.
In contrast, $\BB_s$ is defined in all of $\DD^\pm$ by the
one-dimensional incoming Bernoulli profile.
 Hereafter, we omit the superscript $s$ on variables defined in $\DD$. The Bernoulli quantity $\BB$ is transported from the incoming flow along the streamlines and is continuous across the shocks, while the entropy quantity $\CA$ is transported along the streamlines from its traces on the shocks and is discontinuous across them; the precise formulas are given in Lemma \ref{lem-transport} of Section \ref{sec-mach}. Once $\BB$ is known, $u_1$ is expressed in terms of $\Ub$ by \eqref{eqn-u1}. Thus, with $\BB$ prescribed, all functions of the full state $\CU$ can be written as functions of the reduced triple $\Ub$. Given $\Ub$, we linearize the equations into:
\begin{align}
	\partial_{z_1} \wht + a\, \partial_{z_2} \pht& =   f_1(\BB,\sh,\Ubh),   \label{eqnlhat-w-z1}\\
	\partial_{z_1} \pht - b \, \partial_{z_2} \wht  & =   f_2(\BB,\sh,\Ubh).\label{eqnlhat-p-z1}
\end{align}
We rewrite \eqref{con-RH3} as
\begin{align} \label{eqn-sigmaprime}
	(s^{\pm})' (z_1) = \frac{[p]}{[ u_1 w ]}( z_1, k^{\pm} z_1).
\end{align}
Then we use \eqref{eqn-sigmaprime} to eliminate $	(s^{\pm})'$ in conditions \eqref{con-RH1}, \eqref{con-RH2} and obtain
\begin{align}
	G_1(U^{I,s}, \overline{\CU})& :=  [p] \left[ \frac{1}{\rho u_1 }\right] + [w] [u_1 w] = 0, \label{con-G1}\\
	G_2(U^{I,s}, \overline{\CU})&  :=   [p] \left[ u_1 + \frac{p}{\rho u_1 }\right] + [pw] [u_1 w] = 0.\label{con-G2}
\end{align}

We linearize the conditions above as
\begin{equation}
	\grad_{\overline{\CU}} G_i(U^I_0, \overline{\CU}_0)\cdot \tilde{\Ubh}= 	\grad_{\overline{\CU}} G_i(U^I_0, \overline{\CU}_0)\cdot \Ubh -G_i(U^{I,s}, \overline{\CU}), \label{con-lin-gi}
\end{equation}
denoted by
\begin{equation}
	b_{i1}  \wht  + b_{i2}\pht  +b_{i3} \rhoht  =g_i (U^{I,s}, \overline{\CU}), \label{con-gi}
\end{equation}
where $i=1,2$,
\begin{align}
	(b_{i1},b_{i2},b_{i3}) &=  \grad_{\overline{\CU}} G_i(U^I_0, \overline{\CU}_0),\\
	g_i (U^{I,s}, \overline{\CU})&=  \grad_{\overline{\CU}} G_i(U^I_0, \overline{\CU}_0)\cdot \Ubh -G_i(U^{I,s},\overline{\CU}). \label{eqn-gi}
\end{align}
Using the two conditions \eqref{con-gi} ($i=1,2$) to eliminate $\rhoht $ leads to
\begin{equation} \label{con-g30}
	(b_{11}b_{23} - b_{21} b_{13})\wht + (b_{12}b_{23} - b_{22}b_{13})\pht   =b_{23}g_1 -b_{13}g_2.
\end{equation}
A direct calculation gives
\begin{align*}
	b_{13}&= (p^I_0 - p_0)\left(  \frac{1}{\rho_0^2 u_{10}} +   \frac{\gc p_0}{(\gc -1)\rho_0^3 u_{10}^3}   \right)  < 0,\\
	b_{11}b_{23} - b_{21} b_{13}
	&= ( p^I_0 - p_0)\,  w^I_0\, \frac{(\gc-1)p^I_0+(\gc+1)p_0}{(\gc-1)\rho_0^2}\neq 0 ,
\end{align*}
where $w^I_0=u^I_{20}/u^I_{10}$ and where we used the $x_1$--momentum relation
\begin{equation}\label{con-momentum}
	p^I_0-p_0=\rho_0u_{10}\left(u_{10}-u^I_{10}\right)
\end{equation}
across $S^{\pm}$, which follows from \eqref{RH1} and \eqref{RH2} at the
background configuration. The last inequality holds because $w^I_0\ne0$ --- the
background shocks are not normal --- and $p^I_0\ne p_0$.
Therefore, condition \eqref{con-g30} becomes
\begin{equation} \label{con-g3}
\wht  + \tau\pht  =g_3,
\end{equation}
where
\[
\tau=\frac{b_{12}b_{23} - b_{22}b_{13}}{b_{11}b_{23} - b_{21} b_{13}}, \quad
g_3 = \frac{b_{23} g_1 - b_{13} g_2}{b_{11}b_{23} - b_{21} b_{13}}.
\]
For the fixed background incoming state, the differential of the
shock polar relation is $\mathrm dw+\tau\,\mathrm dp=0$.
Hence $\tau=-\mathrm dw/\mathrm dp$ is precisely the normal-vector
parameter in \eqref{TC}. The nonzero right-hand side in
\eqref{con-g3} includes both the incoming perturbation and the
nonlinear remainder; it does not assert that the finite
perturbation lies on the background tangent line.

Write condition \eqref{con-g3} separately on $S_0^+$ and $S_0^-$:
\begin{align} \label{con-g+}
	\wht(z_1,k^+z_1)  + \tau^+\pht(z_1,k^+z_1)  =g^+(z_1),\\
		\wht(z_1,k^-z_1) + \tau^-\pht(z_1,k^-z_1) = g^-(z_1).\label{con-g-}
\end{align}

Condition \eqref{con-gi} for $i=1$ can be rewritten as
\begin{equation} \label{con-g4}
\rhoht =g_4 - b_2   \wht   -b_3  \pht ,
\end{equation}
where $ g_4 = \frac{g_1}{b_{13}}, b_2= \frac{b_{11}}{b_{13}}, b_3= \frac{b_{12}}{b_{13}} $.

Set
\begin{align}
&M_1 =	\begin{pmatrix}
		0 & a \\
		-b & 0
	\end{pmatrix},
 	T =	\begin{pmatrix}
1 & \tau^+ \\
1 & \tau^-	\end{pmatrix},   \vv=\begin{pmatrix}
v_1\\
v_2  	\end{pmatrix} = T\begin{pmatrix}
\wht\\
\pht	\end{pmatrix},  \label{def-vv}   \\
& M_2= TM_1T^{-1}= \begin{pmatrix}
\frac{ b\tau^+\tau^- +a}{\tau^+ -\tau^-} & \frac{- b(\tau^+)^2 -a}{\tau^+ -\tau^-} \\
\frac{ b(\tau^-)^2 + a}{\tau^+ -\tau^-} &\frac{ -b\tau^+\tau^- -a}{\tau^+ -\tau^-}
\end{pmatrix}.
\end{align}
Multiplying equations \eqref{eqnlhat-w-z1} and \eqref{eqnlhat-p-z1} by $T$ gives
\begin{align} \label{eqn-vv}
\vv_{z_1} +M_2\vv_{z_2} = T\ff,\qquad \ff=(f_1,f_2)^\top.
\end{align}
Equivalently,
\begin{align} \label{eqn-v1z1}
(v_1)_{z_1} -  \tfrac{   b\tau^+\tau^- +a}{b(\tau^-)^2 +a} (v_2)_{z_1} -\tfrac{ab(\tau^+ -\tau^-) }{b(\tau^-)^2 +a} (v_2)_{z_2} &= \tfrac{ b\tau^-(\tau^- -\tau^+)}{b(\tau^-)^2 +a} f_1+     \tfrac{ a(\tau^+ -\tau^-) }{b(\tau^-)^2 +a} f_2  ,\\
(v_1)_{z_2} + \tfrac{  \tau^+-\tau^-}{b(\tau^-)^2 +a} (v_2)_{z_1} -\tfrac{b\tau^+\tau^- +a}{b(\tau^-)^2 +a} (v_2)_{z_2}&= \tfrac{ \tau^+-\tau^-}{b(\tau^-)^2 +a} f_1+     \tfrac{ \tau^- (\tau^+- \tau^- )  }{b(\tau^-)^2 +a} f_2.     \label{eqn-v1z2}
\end{align}

Differentiating \eqref{eqn-v1z2} with respect to $z_1$ and subtracting the $z_2$-derivative of \eqref{eqn-v1z1} eliminates $v_1$ and gives the following divergence-form equation for $v_2$:
\begin{align} \label{eqn-v2}
(a^{ij}(v_2)_{z_i})_{z_j}  = (F^i)_{z_i},
\end{align}
 where
\begin{align}
&	a^{11} =  \tfrac{  \tau^+-\tau^-}{b(\tau^-)^2 +a},  a^{12} =  \tfrac{   b\tau^+\tau^- +a}{b(\tau^-)^2 +a} , 	a^{21} =-\tfrac{b\tau^+\tau^- +a}{b(\tau^-)^2 +a}, 	a^{22} =\tfrac{ab(\tau^+ -\tau^-) }{b(\tau^-)^2 +a},\label{def-aij}\\
& F^1=\tfrac{ \tau^+-\tau^-}{b(\tau^-)^2 +a} f_1+     \tfrac{ \tau^- (\tau^+- \tau^- )  }{b(\tau^-)^2 +a} f_2,  F^2= \tfrac{ b\tau^-(\tau^+ -\tau^-)}{b(\tau^-)^2 +a} f_1 -    \tfrac{ a(\tau^+ -\tau^-) }{b(\tau^-)^2 +a} f_2.\label{def-Fi}
\end{align}
Boundary conditions \eqref{con-g+}, \eqref{con-g-}, and \eqref{con-truncy} give
\begin{align}
&	v_1|_{S^+_0} =  g^+(z_1),\label{con-v1}\\
&		v_2|_{S^-_0} =  g^-(z_1), \quad	v_2|_{\LL} =  g^-(l) .\label{con-v2}
\end{align}
Differentiating \eqref{con-v1} along $S^+_0$ and using equations \eqref{eqn-v1z2} and \eqref{eqn-v1z1} leads to
\begin{align}
 D_{\nnu} v_2|_{S^+_0} = g_5,  \label{con-g5}
\end{align}
 where
\begin{align}
	&\nnu =(\nu_1,\nu_2)^\top,  \quad  D_{\nnu} = \nnu \cdot \grad,\\
&\nu_1	=  b^+\tau^+\tau^- +a^+ - k^+(\tau^+-\tau^-),\\
&\nu_2	= a^+b^+(\tau^+ -\tau^-)+k^+(b^+\tau^+\tau^- +a^+),\\
	&g_5=( b^+(\tau^-)^2 +a^+)((g^+)'-k^+F^1+F^2). \label{def-g5}
\end{align}
Due to the transversality condition \eqref{TC}, $\tau^+-\tau^- \neq0$, we conclude that
\begin{align*}
\nu_1k^+-\nu_2 = -((k^+)^2 +a^+b^+)(\tau^+-\tau^-) \neq0,
\end{align*}
which implies that $\nnu$ is not parallel to the tangential direction $(1,k^+)$. Moreover, if $\tau^+-\tau^- >0$, $\nnu$ is an oblique direction pointing outward. For $\tau^+-\tau^-<0$, multiply both sides of \eqref{con-g5} by $-1$; thus both $\nnu$ and $g_5$ change sign when the outward orientation is chosen.

\section{A mixed boundary value problem} \label{sec-mbvp}

We need to solve the following mixed boundary value problem:

\begin{probc}
\begin{align} \label{eqn-v}
	(a^{ij}v_{z_i})_{z_j}  = (F^i)_{z_i}
\end{align}
  in $\DD$, with boundary conditions
\begin{align}
	&		v|_{S^-_0} =  g^-, \label{con-vs-}\\
	&	v |_{\LL} =  g^-(l) ,\label{con-vl}\\
	& D_{\nnu} v |_{S^+_0} = g_5,  \label{con-vs+}
\end{align}
where $a^{ij}$ are the piecewise constants in \eqref{def-aij}, initially with $\tau^+>\tau^-$, and $\nnu$ is the outward oblique direction in \eqref{con-g5}.
If $\tau^+<\tau^-$, the same analysis applies after multiplying \eqref{eqn-v} and the original oblique condition by $-1$, i.e. replacing
$(a^{ij},F^i,\nnu,g_5)$ by $(-a^{ij},-F^i,-\nnu,-g_5)$.
In this section the common sign is chosen so that the diagonal coefficients are positive and the oblique vector points outward. The reconstruction in Section~\ref{sec-mach} uses the original coefficients of \eqref{eqn-v1z1}--\eqref{eqn-v1z2}.

\end{probc}

Throughout this section, write $L\vf:=(a^{ij}\vf_{z_i})_{z_j}$.

\begin{proposition}\label{weakineq} Suppose that $v \in C^0(\overline{\DD})\cap C^2(\DD^+) \cap C^2(\DD^-) \cap C^1(\DD^+\cup \II) \cap C^1(\DD^-\cup \II)$ is continuous across $\II$. Then
\begin{align} \label{ineqn-v}
	(a^{ij}v_{z_i})_{z_j}  \le (F^i)_{z_i}
\end{align}
 in $\DD$ in the weak sense if and only if
\eqref{ineqn-v} holds pointwise in each $\DD^+$ and $\DD^-$, and $[a^{i2}v_{z_i}]\le [F^2]$ on $\II$.
\end{proposition}
Proposition \ref{weakineq} can be shown easily by integrating by parts and using the definition of inequality \eqref{ineqn-v} in the weak sense.

\begin{proposition}\label{prop-compare}
Suppose $\vf,v\in C^*(\DD)$, where the regularity class is defined
in \eqref{def-Cstar}, and the following differential inequality holds
in the weak sense:
	\begin{align*}
		&(a^{ij}\vf_{z_i})_{z_j}  \le  	(a^{ij}v_{z_i})_{z_j} \quad \text{in}  \  \DD,\\
			&		\vf \ge v \quad \text{on}  \   S^-_0 \cup \LL,  \\
		& \DT_{\nnu} \vf \ge \DT_{\nnu} v\quad \text{on}  \   S^+_0.
	\end{align*}
Then $\vf\ge v$ in $\DD$.
\end{proposition}
\begin{proof}
Let $\chi=v-\vf$. Then $L\chi\ge0$ weakly, $\chi\le0$ on
the Dirichlet boundary and $D_{\nnu}\chi\le0$ on the oblique
boundary. A positive maximum cannot occur on the Dirichlet
boundary, including the endpoints by continuity. The strong maximum principle for weak subsolutions of uniformly elliptic divergence-form equations excludes an interior maximum, including one on $\II$, unless $\chi$ is constant. A positive constant is ruled out by the Dirichlet data. At a relative interior point of $S_0^+$, the coefficients are constant in a one-sided neighbourhood. Hopf's boundary point lemma gives $\partial_n\chi>0$, while the tangential derivative vanishes at a maximum. Since $\nnu\cdot n>0$, this implies $D_{\nnu}\chi>0$, again a contradiction.
Therefore, we conclude $\chi\le 0$ in $\DD$.
\end{proof}

Denote
\begin{align*}
	C^{k,\ga}_{(\gb,\gb')(\gbb)}(\DD)&:= C^{k,\ga}_{(\gb,\gb';\II)(\gbb;\{P^+\})}(\DD^+) \cap C^{k,\ga}_{(\gb,\gb';\II)(\gbb;\{P^-\})}(\DD^-) ,\\
	\|f\|_{k,\ga}^{(\gb,\gb')(\gbb)}&:=	\|f\|_{k,\ga;\DD}^{(\gb,\gb';\II)(\gbb;\PP)},
\end{align*}
where the norm $\|f\|_{k,\ga;\DD}^{(\gb,\gb';\II)(\gbb;\PP)}$ follows the same convention as in \eqref{def-fGO}.

Denote
\begin{align} \label{def-Cstar}
C^*(\DD):={}&C^0(\overline{\DD})\cap C^2(\DD^+)\cap C^2(\DD^-)\nonumber\\
&\cap C^1(\DD^+\cup\II\cup S_0^+)\cap C^1(\DD^-\cup\II).
\end{align}
We state the a priori estimate of solutions to {\bf Problem MBV} as follows.
\begin{lemma} \label{lem-apri}
There exists $\ga_0\in(0,1)$, depending only on the
background coefficients and domain, such that the following
holds for every $0<\ga<\ga_0$. Suppose that $v \in C^*(\DD) $ is a weak solution to {\bf Problem MBV}. Assume that $F^i \in C^{1,\ga}_{(-\ga,1-\ga)(1-\ga)}(\DD) $, $g^- \in C^{2,\ga}_{(-\ga; \{0,l\})}((0,l))$, $g_5 \in C^{1,\ga}_{(1-\ga; \{0,l\})}((0,l))$. Then $v \in C^{2,\ga}_{(-1-\ga,-\ga)(-\ga)}(\DD) $ with the following estimate:
\begin{align} \label{est-v}
	\|v \|_{2,\ga}^{(-1-\ga,-\ga)(-\ga)}  \le C \gk ,
\end{align}
where
\begin{align}	\label{def-kappa}
	\gk :=\sum_{i=1,2}\|F^i\|_{1,\ga}^{(-\ga,1-\ga)(1-\ga)} +\|g^-\|_{2,\ga;(0,l)}^{(-\ga; \{0,l\})} +\|g_5\|_{1,\ga;(0,l)}^{(1-\ga; \{0,l\})},
\end{align}
and $C$ is a positive constant depending on $a,b,\tau^+,\tau^-, k^+,k^-,l,\ga$.
\end{lemma}

\begin{proof}
We first make a piecewise linear change of variables in the weak formulation. In each subdomain put
$\bar\zz=J_\pm\zz$, where $J_\pm=\operatorname{diag}(1,\Lambda_\pm)$ and
$\Lambda_\pm=\sqrt{a^{11}_\pm/a^{22}_\pm}$. Writing $A_\pm=(a^{ij}_\pm)_{i,j=1,2}$, the transformed quantities are
\begin{align*}
 \bar A_\pm&=(\Lambda_\pm)^{-1}J_\pm A_\pm J_\pm^\top,
 &\bar F&=(F^1/\Lambda_\pm,F^2),\\
 \bar\nnu&=J_+\nnu,
 &\bar k^\pm&=\Lambda_\pm k^\pm.
\end{align*}
Thus $\bar a^{11}_\pm=\bar a^{22}_\pm=\sqrt{a^{11}_\pm a^{22}_\pm}$,
while the antisymmetric off-diagonal entries are unchanged. The map agrees on $\II$ from both sides and preserves the weak transmission condition.
Dropping the bars on the coordinates, tensor, data and slopes, write
$a^{11}_\pm=a^{22}_\pm$.
Then $L\vf=a^{11}_\pm\Delta\vf$ in $\DD^\pm$, and, for a continuous piecewise $C^1$ function,
\[
 [a^{i2}\vf_{z_i}]=[a^{22}\vf_{z_2}]+[a^{12}]\vf_{z_1}\quad\hbox{on }\II.
\]
All weighted norms before and after this piecewise linear change of variables are equivalent, with constants depending only on the background data.

We use polar coordinates $(r,\gth)$ centered at $\0$ to define
\begin{align*}
\vf_1(\zz)=	\left\{
	\begin{array}{ll}
	\vf_1^+(r,\gth)=	r^\ga \sin(  \ga^{\frac{1}{2}}\gth+\frac{\pi}{4}),& \gth \ge 0;\\
	\vf_1^-(r,\gth)=	r^\ga \sin( \ga^{\frac{1}{4}}\gth+\frac{\pi}{4}),& \gth < 0.
\end{array}
\right.
\end{align*}
The different frequency parameters $\ga^{1/2}$ and $\ga^{1/4}$ in $\DD^+$ and $\DD^-$ are chosen according to the aperture of the two wedges, so that $\sin(\gl\gth+\pi/4)$ keeps a fixed positive sign in each subdomain and the jump estimate on $\II$ below holds; since $\ga<1$, taking $\ga$ suitably small makes both requirements compatible.

Noticing that $a^{ij}$ are piecewise constant and $a^{12}= -a^{21}$, we see that in each subdomain $\DD^+$ and $\DD^-$, $L = a^{11}\GD$. Therefore, in $\DD^+$ we have
\begin{align*}
L	\vf_1(\zz)&= a^{11}_+\GD 	\vf_1^+(r,\gth)\\
&= a^{11}_+r^{\ga-2} (\ga^2 -\ga) \sin( \ga^{\frac{1}{2}}\gth+\tfrac{\pi}{4}) \\
& \le -c r^{\ga-2}  \qquad (\ga  <\tfrac{1}{8^4} ).
\end{align*}
	Similarly, in $\DD^-$ we have the same estimate
\begin{align}
	L	\vf_1(\zz)& \le -c r^{\ga-2}.  \label{est-vf1}
\end{align}
On $\II$, we have
\begin{align}
&[a^{12}(\vf_1) _{z_1} +a^{22}(\vf_1) _{z_2}	]\nonumber\\
={}& (a^{12}_+ - a^{12}_-)\po_r\vf^+_1(r,0)+\frac{1}{r}\po_\gth(a^{22}_+\vf^+_1- a^{22}_-\vf^-_1)(r,0)\nonumber\\
\le  {}&   |a^{12}_+- a^{12}_-|\ga r^{\ga-1} +(a^{22}_+ \ga^{\frac{1}{2}}-  a^{22}_- \ga^{\frac{1}{4}})r^{\ga-1}\sin(  \tfrac{\pi}{4}) \nonumber\\
\le  {}&  -c r^{\ga-1}     \qquad (\ga \text{ suitably small}). \label{est-vf1-jump}
\end{align}

Denote the tangential and normal components of $\nnu$ by
\begin{align*}
\nu_t &= \nnu \cdot \left(\tfrac{1}{\sqrt{{k^+}^2+1}}, \tfrac{k^+}{\sqrt{{k^+}^2+1}}\right)^\top,\\
 \nu_n &= \nnu \cdot \left(-\tfrac{k^+}{\sqrt{{k^+}^2+1}},\tfrac{1}{\sqrt{{k^+}^2+1}} \right)^\top.
\end{align*}
Estimate $D_{\nnu}\vf_1$ on $S^+_0$ as follows:
\begin{align*}
D_{\nnu}\vf_1 (r, \go^+) &= (\nu_t \po_r + \nu_n\frac{1}{r}\po_\gth)\vf_1  (r, \go^+)\\
&=r^{\ga-1}\left(\nu_t\ga\sin( \ga^{\frac{1}{2}}\go^+ +\tfrac{\pi}{4}) +\nu_n \ga^{\frac{1}{2}} \cos(\ga^{\frac{1}{2}}\go^+ +\tfrac{\pi}{4})\right)\\
&\ge c r^{\ga-1}.
\end{align*}

To control the singularity at $P^0$, we define $\vf_2$ using polar coordinates $(r_*, \gth_*)$ centered at $P^0$:
\begin{align*}
	\vf_2(\zz)=	\left\{
	\begin{array}{ll}
	r_*^\ga \sin(\ga^{\frac{1}{4}}( \gth_*-\pi)+\frac{\pi}{4} ),& \frac{\pi}{2} \le \gth_* \le \pi;\\
	 	r_*^\ga \sin(\ga^{\frac{1}{2}}( \gth_*-\pi)+\frac{\pi}{4} ),& \pi <\gth_*\le\frac{3}{2} \pi  .
	\end{array}
	\right.
\end{align*}

To control the singularity at $P^+$ and $P^-$, we define $\vf_3$ and $\vf_4$, using polar coordinates $(r_+, \gth_+)$ centered at $P^+$ and $(r_-, \gth_-)$ centered at $P^-$, respectively as follows:
\begin{align*}
	\vf_3(\zz)&=
		r_+^\ga \cos(\ga^{\frac{1}{2}} \gth_+  +\tfrac{\pi}{4}),\\
\vf_4(\zz)&=
		r_-^\ga \sin(\ga^{\frac{1}{2}}  \gth_-  +\tfrac{\pi}{4}).
\end{align*}
Denote
\begin{align*}
 l^*:=\frac18\min\left\{l,lk^+,-lk^-,
 \frac{lk^+}{\sqrt{1+(k^+)^2}},
 \frac{-lk^-}{\sqrt{1+(k^-)^2}}\right\}.
\end{align*}
In the definitions of $\vf_3$ and $\vf_4$, the angles are
measured counterclockwise from the positive $z_1$-axis, with
$\pi+\arctan k^+\le\gth_+\le3\pi/2$ near $P^+$ and
$\pi/2\le\gth_-\le\pi+\arctan k^-$ near $P^-$.
In particular, on $S_0^+$ near $P^+$, its outward normal is
$-e_{\gth_+}$, and
\begin{align*}
 D_{\nnu}\vf_3
 &=r_+^{\ga-1}\left\{
 -\ga\nu_t\cos(\ga^{1/2}\gth_++\pi/4)\right.\\
 &\hspace{24mm}\left.
 +\ga^{1/2}\nu_n\sin(\ga^{1/2}\gth_++\pi/4)\right\}
 \ge c r_+^{\ga-1},
\end{align*}
for $\ga$ sufficiently small. The choice of $l^*$ keeps the
corner neighbourhoods disjoint and keeps the neighbourhood of
$P^0$ away from both shock sides.

To control the mixed singularity near the contact discontinuity and its endpoints, set, for $\zz\in\DD^\pm$,
\begin{align}\label{def-H}
	H(\zz):= -\Big(\frac{1}{r}+\frac{1}{r_*}\Big)|z_2|^{1+\ga},
\end{align}
where $r=|\zz|$ and $r_*=|\zz-P^0|$. A direct computation gives
\begin{align}\label{est-H}
	L H &= -a^{11}_\pm\ga(1+\ga)\Big(\frac{1}{r}+\frac{1}{r_*}\Big)|z_2|^{\ga-1}
	 +a^{11}_\pm(1+2\ga)\Big(\frac{|z_2|^{1+\ga}}{r^{3}}+\frac{|z_2|^{1+\ga}}{r_*^{3}}\Big)\nonumber\\
	&\le -c\Big(\frac{1}{r}+\frac{1}{r_*}\Big)|z_2|^{\ga-1}+C(r^{\ga-2}+r_*^{\ga-2}) ,
\end{align}
where we used $|z_2|\le r,r_*$ on $\DD$; moreover $H=0$ and $\nabla H=0$ on $\II\setminus\{\O,P^0\}$.
Set
\begin{align*}
&	\widehat\vf_5 := C\gk \Big(C_1 \vf_1 + C_2\eta(\tfrac{r_*}{l^*}) \vf_2+\eta(\tfrac{r_+}{l^*} )\vf_3+\eta(\tfrac{r_-}{l^*})\vf_4+H\Big),
\\
& \vf_5:=|g^-(0)|+\widehat\vf_5 .
\end{align*}

{\bf Verification of the barrier $\vf_5$.}
The argument below is similar to that in \cite[Step 2]{ADV1} except for an additional estimate across the contact discontinuity $\II$.  We only give the details for $\vf_2$; the estimates for $\vf_3, \vf_4$ are analogous. As computed
above for $\vf_1$ (see \eqref{est-vf1}), we have that in each of $\DD^+\cap B_{l^*}(P^0)$ and $\DD^-\cap B_{l^*}(P^0)$,
\begin{align}\label{est-vf2}
	&L \vf_2\le -c\,r_*^{\ga-2},  
\end{align}
and on $\II\cap B_{l^*}(P^0)$,
\begin{align}\label{est-vf2-jump1}
 &[a^{12}(\vf_2) _{z_1} +a^{22}(\vf_2) _{z_2}	]
\le   -c r_*^{\ga-1}.
\end{align}

The cut-off  $\eta(r_*/l^*)$ localizes the singularity of  $\vf_2$ in the following sense:
\begin{align}\label{est-vf2-loc}
&	\|\eta(r_*/l^*)\vf_2\|_{C^2(\DD^{\pm}\cap B_{2l^*}(P^0)\setminus B_{l^*}(P^0))}\le C_0,\\
&	\sup_{\II\cap\{l^*<r_*<2l^*\}}
 \left|\left[a^{i2}\partial_{z_i}
       \big(\eta(r_*/l^*)\vf_2\big)\right]\right|\le C_0. \label{est-vf2-jump2}
\end{align}
Similar estimates to \eqref{est-vf2} and \eqref{est-vf2-loc} can be obtained for $\vf_3,\vf_4$. We note that, by the choice of $l^*$, the supports of the cut-offs $\eta(r_*/l^*)$, $\eta(r_+/l^*)$ and $\eta(r_-/l^*)$ are pairwise disjoint and do not contain $\O$; hence at every point of $\DD$ at most one of the localized terms is singular. 
The role of the singular factor in \eqref{def-H} is the following. The data entering \eqref{def-kappa} satisfy, by \eqref{def-normsemi} with $k=1$,
\begin{align}\label{est-divF}
	|(F^i)_{z_i}|\le C\gk\,\Big(\frac{1}{r}+\frac{1}{r_*}\Big)|z_2|^{\ga-1}
	+C\gk\,\big(r_+^{\ga-2}+r_-^{\ga-2}\big).
\end{align}
The corner barriers and $-|z_2|^{1+\ga}$ alone do not control
the product $r^{-1}|z_2|^{\ga-1}$ near $\O$ or the corresponding
term near $P^0$. The function $H$ in \eqref{def-H} supplies these
terms through \eqref{est-H}.

We now specify the constants, with $\ga$ already fixed sufficiently small.
Let $C_H$ bound the positive remainder in \eqref{est-H}, and write
$L\vf_2\le-c_2r_*^{\ga-2}$ on $B_{l^*}(P^0)$.
First choose $C_2$ so that $C_2c_2\ge C_H+1$.
Then choose $C_1$ sufficiently large, depending on $C_2$, to absorb the
$r^{\ga-2}$ remainder from $H$, all cut-off errors, and the bounded
boundary and flux errors. Outside $B_{l^*}(P^0)$ the remaining
$r_*^{\ga-2}$ term is bounded and is absorbed in the same way.
Since $r\le\operatorname{diam}\DD$, the negative term
$-C_1c r^{\ga-2}$ also controls every bounded remainder. Consequently,
in each of $\DD^+$ and $\DD^-$,
\begin{align*}
 L\widehat\vf_5
 &\le -C\gk c\left[
 \left(\frac1r+\frac1{r_*}\right)|z_2|^{\ga-1}
 +r_+^{\ga-2}+r_-^{\ga-2}\right]\\
 &\le (F^i)_{z_i},
\end{align*}
where the last inequality follows from \eqref{est-divF} after choosing the overall multiplier $C$ sufficiently large.
The same choice of $C_1$ also makes the bracket defining $\widehat\vf_5$ nonnegative, since $|H|\le Cr^\ga$ and the localized terms are nonnegative.

It is easy to see that   $\widehat\vf_5\ge g^- - g^-(0)$  on $S^-_0\cup\LL$  for
$C$ large: on $S^-_0$ this follows from $|g^-(z_1)-g^-(0)|\le C\gk r^\ga$, which is a consequence of the weighted H\"older continuity of $g^-$ up to $\O$, while on $\LL$ we use $C_1\vf_1\ge c>0$. On $S^+_0$, the estimates $D_{\nnu}\vf_1\ge c\,r^{\ga-1}$ and $|D_{\nnu}H|\le Cr^{\ga-1}$ obtained
above, together with $D_{\nnu}\vf_3\ge c\,r_+^{\ga-1}$ near $P^+$ and the boundedness of
the remaining terms, give
\begin{align*}
	D_{\nnu}\widehat\vf_5\ge C\gk\big(C_1c\,r^{\ga-1} +  c r_+^{\ga-1} -C_0\big)\ge g_5
\end{align*}
for $C_1$ suitably large and then $C$ large.

Since \eqref{eqn-v} is understood in the weak sense, Proposition \ref{weakineq}
requires in addition the jump condition
\begin{align}\label{est-vf5-jump}
	[a^{i2}(\vf_5)_{z_i}]\le [F^2]\qquad\text{on }\ \II .
\end{align}
By \eqref{est-vf1-jump}, \eqref{est-vf2-jump1}, and \eqref{est-vf2-jump2}, we obtain that on $\II$,
\begin{align*}
	[a^{i2}(\vf_5)_{z_i}]\le C\gk\big(-C_1c\,r^{\ga-1} - cr_*^{\ga-1} +C_0 \big)\le [F^2]
\end{align*}
for $C_1$ suitably large and then $C$ large. Therefore \eqref{est-vf5-jump} holds and, by
Proposition \ref{weakineq}, $\vf_5$ satisfies the differential inequality in the weak
sense.

Here and below the Dirichlet set is
$\GG_D:=\overline{S_0^-\cup\LL}$, and its continuous boundary
value $h_D$ is $g^-(z_1)$ on $S_0^-$ and $g^-(l)$ on $\LL$.
All the estimates of this verification remain valid when
$(F^i)_{z_i}$, $g^--g^-(0)$, $g_5$ and $[F^2]$ are replaced by their negatives. Hence,
 we deduce  that
$\widehat\vf_5\ge|g^--g^-(0)|$ on $S^-_0\cup\LL$ and $D_{\nnu}\widehat\vf_5\ge|g_5|$ on $S^+_0$. In particular $\widehat\vf_5\ge0$ on $\GG_D$; hence, setting
$$M_D:=\max_{\GG_D}h_D,\qquad m_D:=\min_{\GG_D}h_D,$$
the functions $M_D+\widehat\vf_5$ and $m_D-\widehat\vf_5$ are, respectively, a supersolution and a subsolution of {\bf Problem
MBV} in the sense of Definition \ref{def-subsol} --- the differential inequality, the jump condition on $\II$ and the oblique derivative inequality are those verified above, while the boundary constraints \eqref{def-sub-D} hold because $\widehat\vf_5\ge0$ there. This justifies the corresponding
statement in the proof of Lemma \ref{lem-exist}.

Therefore, by Proposition \ref{prop-compare}, we conclude that
\begin{align} \label{est-v-C0}
	|v(\zz)-v(\O)| \le \widehat\vf_5 (\zz) \le C\gk r^\ga,
	\qquad\text{and}\qquad |v(\zz)| \le \vf_5 (\zz) \le C\gk ,
\end{align}
where we used $v(\O)=g^-(0)$, which follows from the continuity of $v$ on $\overline\DD$ and the Dirichlet condition on $S^-_0$.
To obtain the required weights at the other three vertices, we also
estimate the oscillation of $v$ near each vertex. At $P^0$,
apply comparison in $\DD\cap B_{l^*}(P^0)$ to
$v-g^-(l)$, using
\[
 C\gk\big(K\vf_2-r_*^{-1}|z_2|^{1+\ga}\big).
\]
Choose $K$ large first. The computation in \eqref{est-H}, now
with only $r_*^{-1}$, and \eqref{est-vf2-jump1} control the
interior source and the flux jump; the barrier is positive and
dominates the solution on the circular boundary by
\eqref{est-v-C0}. On the straight exit boundary the Dirichlet
difference is zero. At $P^+$ and $P^-$ use respectively
$C\gk\vf_3$ and $C\gk\vf_4$ in the corresponding local
domains. At $P^-$ the nonconstant Dirichlet difference obeys
$|g^-(z_1)-g^-(l)|\le C\gk|\zz-P^-|^\ga$.
The interior, oblique and Dirichlet inequalities then give, for each
$P\in\{P^0,P^+,P^-\}$,
\begin{equation}\label{est-other-corners}
 |v(\zz)-g^-(l)|\le C\gk|\zz-P|^\ga,
 \qquad \zz\in\DD\cap B_{l^*/2}(P).
\end{equation}
The same barriers apply to $v-g^-(l)$ and its negative. We combine
these oscillation bounds with local Schauder estimates and estimates
across the contact discontinuity.

\smallskip

\noindent {\bf Corner estimates with weight index $(-\ga)$.} We only take corner $\O$ as an example. Denote
\begin{align*}
&\go^{\pm}:= \arctan k^{\pm},&& \go^*:= \tfrac{1}{8} \min(\go^+,-\go^-).
\end{align*}
Fix $\zz^0\in\overline{\DD}$ with polar coordinates
$(r_0,\gth_0)$ and $0<r_0<l^*/4$, so the local neighbourhoods
below do not meet the exit boundary. For this point, we consider four angle ranges for $\gth_0$:
\begin{align*}
&I_1 = (\go^-+\go^*, -\go^*)\cup (\go^*,\go^+ - \go^*),&& I_2=  [\go^-, \go^-+\go^*]\\
&I_3 = [ \go^+ -\go^*, \go^+],&& I_4=  [-\go^*, \go^*].
\end{align*}
Set
\begin{align*}
&	B_1= B_{\rb}(\zz^0)\cap \DD,&& B_2= B_{2\rb}(\zz^0)\cap \DD , &&\Upsilon=\overline{B_2}\cap \po \DD,
\end{align*}
	where
\begin{align*}
\rb = 	\begin{cases}
 \frac{r_0}{4} \sin \go^*,  & \gth_0 \in I_1, \\
2 r_0 \sin \go^* ,&  \gth_0 \in I_2\cup I_3\cup I_4.
\end{cases}
\end{align*}

We first use $\widehat\vf_5$ as the barrier function for $\vh(\zz):= v(\zz) - v(\0)$ and obtain
\begin{align} \label{est-v-rga}
	|\vh(\zz) |   \le C\gk r^\ga .
\end{align}
Let $G$ be a bounded open domain and let $f$ be defined on $G$. Set $d:=\operatorname{diam}G$.
We use the following scaled H\"older norms (\textit{cf.} page 53 in \cite{gt}):
\begin{align*}
	\| f\|'_{k; G} &= \sum_{j=0}^k d^j[ f]_{j,0;G},\\
\| f\|'_{k,\ga;G} &=\| f\|'_{k; G}+ d^{k+\ga}[ f]_{k,\ga;G}.
\end{align*}
If $\gth_0 \in I_1$, we apply the interior Schauder estimate (\textit{cf.} Corollary 6.3, Theorem 6.26, in \cite{gt}) to obtain
\begin{align} \label{est-vint-2ga}
	\|\vh  \|'_{2,\ga; B_1} \le C(\| \hat{v} \|_{0,0;B_2}+  	\rb^2 \| ( F^i )_{z_i} \|'_{0,\ga; B_2} )   \le C\gk \rb^\ga.
\end{align}
For $\gth_0 \in I_2$ or $\gth_0 \in I_3$, Schauder boundary estimates (\textit{cf.} Lemma 6.4, Lemma 6.29 in \cite{gt}) give rise to
\begin{align} \label{est-vbd-2ga}
 \|\vh\|'_{2,\ga;B_1}
 &\le C\left(\|\vh\|_{0,0;B_2}
 +\rb^2\|(F^i)_{z_i}\|'_{0,\ga;B_2}\right)\nonumber\\
 &\quad+C\begin{cases}
 \|g^--g^-(0)\|'_{2,\ga;\Upsilon},&\gth_0\in I_2,\\
 \rb\|g_5\|'_{1,\ga;\Upsilon},&\gth_0\in I_3
 \end{cases},
\end{align}
and in both cases the right-hand side is bounded by $C\gk \rb^\ga$.

\smallskip

\noindent {\bf Boundary estimates near $\II$ with weight index $(-1-\ga)$.} When $\gth_0\in I_4$, the local neighbourhoods meet the contact discontinuity $\II$, and we seek piecewise $C^{1,\ga}$ regularity up to $\II$. We first estimate the first derivatives up to $\II$ and then the second derivatives in each subdomain.

First, we will derive piecewise $C^{1,\ga}$ estimates based on Lemma 2.2 in \cite{ChenDeng}.
\begin{remark}[Nonsymmetric coefficients]\label{rem-nonsymmetric}
Although symmetry is assumed in Proposition~3.2 of
\cite{LiVogelius}, it is not essential to the argument underlying
\cite[Lemma~2.2]{ChenDeng}. Thus this lemma applies to the nonsymmetric
coefficients considered here; related estimates without a symmetry
assumption can also be found in \cite{LiNirenberg}.
\end{remark}

Denote $B_n^\pm:=B_n\cap\DD^\pm$, $n=1,2$.
The weighted assumptions on $F$ imply
\[
 \sum_{i=1,2}\sum_{\sigma=+,-}
 \|F^i\|'_{0,\ga;B_2^\sigma}
 \le C\gk r_0^{\ga-1}.
\]
Indeed, $|F|\le C\gk r_0^{\ga-1}$ and
$|DF|\le C\gk r_0^{-1}|z_2|^{\ga-1}$ there.
Integration in the normal direction gives one-sided traces on
$\II$. For two points at distance $d$, move them, if necessary,
to a common distance comparable to $d$ from $\II$ and integrate
along the resulting three segments. This yields the one-sided
$C^\ga$ bound $C\gk r_0^{-1}$, and hence the displayed scaled
estimate. No H\"older quotient is taken between the two traces.
The transmission estimate can be applied on balls centred on
$\II$ and on interior balls covering $B_1$, with radii comparable
to $\rb$.
Under $\zz=\zz^0+\rb\yy$, set
$V(\yy)=\vh(\zz^0+\rb\yy)$ and
$\mathcal F^i(\yy)=\rb F^i(\zz^0+\rb\yy)$. The equation is
\[
 \partial_{y_j}(a^{ij}\partial_{y_i}V)
   =\partial_{y_i}\mathcal F^i
\]
in the rescaled ball, with the same continuity and conormal-flux
conditions on its horizontal interface. Applying Lemma~2.2,
estimate~(2.7), of \cite{ChenDeng}, with coefficient matrix
$(a^{ji})$, gives
\begin{align} \label{est-vcontact-1ga}
	\sum_{j=+,-}\|\vh  \|'_{1,\ga; B_1^j} \le C(\| \hat{v} \|_{0,0;B_2}+  	\rb \sum_{\begin{subarray}{c}
		i=1,2,\\
	j=+,-
\end{subarray}} \| ( F^i ) \|'_{0,\ga; B_2^j} )   \le C\gk \rb^\ga.
\end{align}

To derive the one-sided $C^{2,\ga}$ estimates, take $\zz^0\notin\II$ and set
\begin{align*}
&\db := \frac{r_0}{4} |\sin \gth_0|,\quad\bar{B}_n := B_{n\db}(\zz^0), \quad n=1,2 ,\\
&\vb(\zz):= v(\zz)-v(\zz^0)- Dv(\zz^0)\cdot(\zz-\zz^0) .
\end{align*}
Since the coefficients are constant in $\bar B_2$, subtracting the affine function preserves \eqref{eqn-v}. The interior Schauder estimate therefore gives
\begin{align} \label{est-vcontact-2ga}
	\|\vb  \|'_{2,\ga; \bar{B}_1} \le C(\| \vb \|_{0,0;\bar{B}_2}+  	\db^2 \| ( F^i )_{z_i} \|'_{0,\ga; \bar{B}_2} )   \le C\gk \rb^{-1}\db^{1+\ga}.
\end{align}
These estimates establish the required weights near $\O$. At
$P^0$ use \eqref{est-other-corners} in the same argument, with
$r_0=|\zz^0-P^0|$; the two straight boundary portions carry
Dirichlet data. At $P^+$ and $P^-$ the interface is absent from
the chosen neighbourhoods, and the corresponding interior,
Dirichlet and oblique estimates give the weight $-\ga$.
Away from these four points, a finite covering by ordinary
interior, boundary and interface neighbourhoods gives uniform
estimates. For the global weighted H\"older seminorm, pairs
closer than a fixed fraction of the local radius are covered by
the local estimates. For the remaining pairs, the weighted
pointwise derivative bounds and the lower bound on their
separation give the same quotient bound. The covering radii and
constants depend only on the background domain. This proves \eqref{est-v}.\qedhere
\end{proof}

We will show the existence of solutions to {\bf Problem MBV} by Perron's method. We introduce the following definitions, similar to those in \cite{Lieberman1,Lieberman2}.
\begin{definition}
We say that {\bf Problem MBV} is locally solvable if, for every
$\zz^*\in\DD\cup S_0^+$, there is a neighbourhood $\NN$, relatively
open in $\DD\cup S_0^+$, for which the local equation and oblique
condition have a unique solution for every continuous Dirichlet
value on the artificial boundary. Write
\begin{align*}
 \NN^\pm&=\NN\cap\DD^\pm,\qquad
 \po^b\NN=\po\NN\cap S_0^+,\\
 \po^i\NN&=\po\NN\setminus\po^b\NN,\qquad
 K_\pm=(\overline{\NN^\pm}\cap(\DD\cup S_0^+))
              \setminus\overline{\po^i\NN}.
\end{align*}
The solution class is
\[
 C^*(\NN)=C^0(\overline\NN)\cap C^2(\NN^+)\cap C^2(\NN^-)
              \cap C^1_{\mathrm{loc}}(K_+)\cap C^1_{\mathrm{loc}}(K_-).
\]
Here $C^1_{\mathrm{loc}}(K_\pm)$ means one-sided $C^1$ regularity
on every compact subset, including the physical oblique boundary
and the interior of the contact discontinuity. No derivative
regularity is required at the artificial-boundary junctions.
\end{definition}
\begin{remark}
In the definition above, if $\zz^*\notin \II$, we may choose a suitably small neighborhood $\NN$, so that either $\NN^+ $ or $\NN^- $ is empty. Similarly, if $\zz^*\notin S_0^+$, $\po^b \NN$ may be empty.
\end{remark}

\begin{proposition}\label{prop-solvable}
Under the same assumptions on $F^i, g^-, g_5$ as in Lemma \ref{lem-apri}, {\bf Problem MBV} is locally solvable.
\end{proposition}
\begin{proof}
Away from $\II$ choose a small interior neighbourhood, or a local
mixed-boundary neighbourhood at $S_0^+$, and use the local
Dirichlet or oblique theory; see \cite{gt,Lieberman2}. At
$\zz^*\in\II$ choose a ball whose closure avoids all four
vertices. Here $F$ is bounded and piecewise $C^\ga$. For smooth
Dirichlet data, the coercive weak formulation gives a solution
in $H^1$; the antisymmetric part of the coefficient matrix does
not affect coercivity. The interface estimate of
\cite[Lemma 2.2]{ChenDeng} gives piecewise $C^{1,\ga}$
regularity on smaller balls. Since $F$ is $C^{1,\ga}$ in each
open subdomain, interior Schauder estimates also give
$C^{2,\ga}$ regularity there. For continuous artificial-boundary
data, approximate uniformly by smooth data. Comparison makes
the solutions uniformly Cauchy, and local estimates pass the
equation to the limit. Boundary barriers give the prescribed
continuous trace. This argument gives no $C^{2,\ga}$ assertion
across the interface.

\end{proof}

\begin{definition}[Subsolutions and supersolutions]\label{def-subsol}
We assume that {\bf Problem MBV} is locally solvable. Let $\GG_D:=\overline{S^-_0\cup\LL}$ and let
\begin{align}\label{def-hD}
	h_D(\zz):=\begin{cases} g^-(z_1), &\zz=(z_1,k^-z_1)\in S^-_0,\\ g^-(l), &\zz\in\LL. \end{cases}
\end{align}
We say that $\vf\in C^0(\overline\DD)$ is a {\bf subsolution} to \eqref{eqn-v}--\eqref{con-vs+} if
\begin{align}\label{def-sub-D}
	\vf\le h_D \quad\text{on }\ \GG_D ,
\end{align}
and if for any $\zz^* \in \DD\cup S_0^+$, for every sufficiently small admissible neighbourhood $\NN$ of $\zz^*$, any solution $v\in C^*(\NN)$ to the problem $L v = (F^i)_{z_i}$ in $\NN$, $D_{\nnu}v = g_5$ on $\po^b \NN $ and $v\ge \vf$ on $\po^i \NN $ satisfies $\vf \le v$ in $\NN$. Denote the set of all subsolutions by $\Phi^-$. Reversing the inequalities in \eqref{def-sub-D} and in the comparison statement, we define supersolutions and the set $\Phi^+$ in the same manner.
\end{definition}

We use the following properties of subsolutions and supersolutions.
\begin{enumerate}
	\item If $v_1,v_2 \in \Phi^-$, then $\max(v_1,v_2) \in \Phi^-$.

	\item If $v_1 \in \Phi^-, v_2 \in \Phi^+$, then $v_1 \le v_2$.

	\item Assume that $v \in \Phi^-$. Given $\zz^* \in \DD\cup S_0^+$, by the local solvability property, let $\NN$ be an associated neighbourhood. Define $\vb$ as follows: $\vb(\zz) = v(\zz)$ for $\zz \in \DD \backslash \NN$, $L \vb = (F^i)_{z_i}$ in $\NN$, $D_{\nnu}\vb = g_5$ on $\po^b\NN$, and
$\vb=v$ on the artificial boundary $\po^i\NN$. Then
$\vb\in\Phi^-$.
 We call $\vb$ a {\bf lift} of $v$.

	\item Suppose that $\{v_k\}\subset C^*(\NN)$ is uniformly
bounded in $C^0(\overline\NN)$ and satisfies
 $L v_k = (F^i)_{z_i}$ in $\NN$, $D_{\nnu}v_k = g_5$ on $\po^b \NN $. Then there exists a convergent subsequence $\{v_{k_j}\}$, such that $v=\lim\limits_{j\to \infty}v_{k_j}$ is a solution to the same problem $L v = (F^i)_{z_i}$ in $\NN$, $D_{\nnu}v = g_5$ on $\po^b \NN $, and $v \in C^2(\NN^+) \cap C^2(\NN^-) \cap C^1_{\mathrm{loc}}(K_+)\cap C^1_{\mathrm{loc}}(K_-)$.
\end{enumerate}
Properties 1 and 3 can be verified easily. We now prove Property 2, since the comparison requires the Dirichlet constraint \eqref{def-sub-D}.

\medskip
\noindent{\bf Proof of Property 2.} Let $v_1\in\Phi^-$, $v_2\in\Phi^+$ and suppose that
$M:=\max_{\overline\DD}(v_1-v_2)>0$. The Dirichlet constraints imply that the compact contact set
$K:=\{v_1-v_2=M\}$ is disjoint from $\GG_D$.
Choose $\zz^*\in K$ and a sufficiently small admissible neighbourhood $\NN$ of $\zz^*$,
with $\overline\NN\cap\GG_D=\emptyset$, and let $w^-,w^+$ be the local solutions with
artificial-boundary values $v_1,v_2$, respectively, and common oblique data $g_5$.
The defining comparison properties give
\[
 v_1\le w^-,\qquad w^+\le v_2\quad\hbox{in }\NN.
\]
Thus $W:=w^--w^+$ solves the homogeneous transmission problem, with
$D_{\nnu}W=0$ on $\po^b\NN$ and $W=v_1-v_2\le M$ on $\po^i\NN$.
The maximum principle yields $W\le M$, whereas
\[
 M=v_1(\zz^*)-v_2(\zz^*)\le W(\zz^*).
\]
The strong maximum principle, including its weak transmission version at $\II$,
and the boundary point lemma at $S_0^+$ imply $W\equiv M$ in $\NN$.
Consequently $v_1-v_2=M$ on $\po^i\NN$.
At an interior contact point we may use all sufficiently small balls centred there;
at an oblique-boundary contact point we use the corresponding family of local
mixed-boundary neighbourhoods. Their artificial boundaries sweep out a neighbourhood
of the point (including its boundary by continuity). Hence $K$ is relatively open in
$\DD\cup S_0^+$. It is also relatively closed and nonempty, so connectedness implies
$K=\DD\cup S_0^+$. Continuity then contradicts the Dirichlet constraints.
Therefore $v_1\le v_2$. $\square$

\medskip
In Property 4 the convergence is local, up to compact portions
of the physical oblique boundary and separately on the two sides
of $\II$: it is $C^2$ in each open subdomain and $C^1$ up to
these boundary portions. The limit has the same local regularity
as in the definition above; no convergence at the artificial
boundary is asserted. To obtain this subsequence, apply interior
and oblique Schauder estimates away from $\II$, and the
piecewise $C^{1,\ga}$ estimate of \cite[Lemma 2.2]{ChenDeng}
near $\II$, followed by a diagonal argument. In particular, the
weak transmission condition and the oblique condition pass to
the limit. This uses local estimates only, not the global
estimate of Lemma~\ref{lem-apri} on an arbitrary neighbourhood.

We can now establish existence and uniqueness for {\bf Problem MBV}.
\begin{lemma} \label{lem-exist}
	Under the assumptions on $F^i,g^-,g_5$ in Lemma~\ref{lem-apri}, there exists a unique weak solution $v\in C^*(\DD)\cap C^{2,\ga}_{(-1-\ga,-\ga)(-\ga)}(\DD)$ to {\bf Problem MBV}, and it satisfies \eqref{est-v}.
\end{lemma}

\begin{proof}
The uniqueness of solutions follows from Proposition \ref{prop-compare}. With the a priori estimate in Lemma \ref{lem-apri}, it suffices to show the existence of solutions in the space $ C^*(\DD)$.

Set $M_D:=\max_{\GG_D}h_D$ and $m_D:=\min_{\GG_D}h_D$. By Lemma \ref{lem-apri}, $\widehat\vf_5\ge0$ on $\GG_D$, and $\widehat\vf_5$ satisfies the weak differential inequality, the jump condition on $\II$ and the oblique derivative inequality; hence $\overline\vf:=M_D+\widehat\vf_5$ is a supersolution and $\underline\vf:=m_D-\widehat\vf_5$ is a subsolution in the sense of Definition \ref{def-subsol}. Thus the set of subsolutions $\Phi^-$ is not empty and, by Property 2, every $\vf\in\Phi^-$ is bounded above by $\overline\vf$. Set
$$
v(\zz):= \sup_{\vf \in \Phi^-} \vf (\zz),
$$
for any $\zz \in \overline{\DD}$. Next, we will show that $v$ is a weak solution to the problem.

Choose any $\zz^* \in \DD\cup S_0^+$. By the definition of $v$, we can choose a sequence $\{v_k \} \subset \Phi^-$, such that $\lim\limits_{k\to \infty} v_k(\zz^*)= v(\zz^*)$. Obviously, by Property 2 above, $\{v_k \}$ is bounded above by $\overline\vf$. Without loss of generality, we may assume that $\{v_k \} $ is uniformly bounded below ( otherwise, we may replace $\{v_k \}$ with $\{\max(v_k, \underline\vf) \}$, which belongs to $\Phi^-$ by Property 1 ). By the local solvability of {\bf Problem MBV} and Property 3, there is a neighbourhood $\NN$ and a uniformly bounded sequence of lifts $\{\vb_k\}$. By Property 4, there exists a convergent subsequence $\{\vb_{k_j} \} \to \vb $, where $\vb\in C^2(\NN^+) \cap C^2(\NN^-) \cap C^1_{\mathrm{loc}}(K_+)\cap C^1_{\mathrm{loc}}(K_-)$ is a solution to the same local problem.

We claim that $\vb(\zz)=v(\zz)$ throughout $\NN$, including
the interface and the physical oblique boundary.
 Obviously, $\vb \le v $ in $\NN$ and $\vb(\zz^*) = v (\zz^*)$. If the claim is false, we may assume that there exists $\zz^0\in\NN$,
 such that $\vb(\zz^0) <v(\zz^0)$. Then there exists $v^0 \in \Phi^-$, such that $\vb(\zz^0) <v^0(\zz^0) $. Let $v^0_j:=\max(\vb_{k_j},v^0)$ and denote its lift by $\vb^0_j$. We again choose a subsequence of $\{\vb^0_j\} $ converging to $\vh$ with the same properties as $\vb$. The fact that $\vb \le \vh $ in $\NN$ implies that $\vh - \vb $ achieves its minimum at $\zz^*$, which lies in $\DD\setminus\II$, or on $\II$, or on $S_0^+$. Noticing the condition
$D_{\nnu}\vb = D_{\nnu} \vh $ on $\po^b \NN$, the strong maximum principle --- applied on each side of $\II$, and across $\II$ since $\vh-\vb$ is a weak solution of the homogeneous equation there --- in the first two cases, and the boundary point lemma for the oblique derivative operator in the third case, give that $\vh = \vb$ in $\NN$, which contradicts $\vb(\zz^0)<\vh(\zz^0)$.

Therefore, by choosing arbitrary $\zz^*$, we conclude that $v\in C^2(\DD^+) \cap C^2(\DD^-) \cap C^1(\DD^+\cup \II \cup S^+_0)\cap C^1(\DD^-\cup \II)$ satisfies equation \eqref{eqn-v} and the boundary condition \eqref{con-vs+}.

At a Dirichlet point $\zz^*$, use the local corner barriers
constructed above, or a smooth-boundary barrier away from the
vertices. To apply them to the Perron envelope, compare every
subsolution with $h_D(\zz^*)+\epsilon+K B$ on a small
neighbourhood. For the reverse inequality, patch
$h_D(\zz^*)-\epsilon-K B$ to the global subsolution by taking
their maximum, choosing $K$ so the global subsolution dominates
on the artificial boundary. The pasting property follows from
local comparison. Letting $\epsilon\downarrow0$ proves the
boundary value and the local estimate
\begin{align*}
|v(\zz) - h_D(\zz^*)| \le C|\zz - \zz^*|^\ga.
\end{align*}
 This implies that $v$ is continuous up to $\po \DD \backslash S^+_0$ and boundary conditions
 \eqref{con-vs-} and \eqref{con-vl} hold. Thus, $v \in C^*(\DD)$ is a solution to {\bf Problem MBV}.
\end{proof}

\section{Solution to the Mach configuration problem} \label{sec-mach}
We introduce convex sets for iteration as follows.
 \begin{align}\label{def-sapcesigma}
	\begin{cases}
		\Sigma_1^\gd &:= \{ v: \| v\|_{2,\ga}^{(-1-\ga,-\ga)(-\ga)} \le \gd \},\\
		\Sigma_2^\gd &:= \{ v: \| v\|_{2,\ga}^{(-\ga,- \ga)(-\ga)} \le \gd\},\\
		\Sigma_3^\gd &:= \{ \sh: \sh(0)=0,\ \| \sh\|_{3,\ga;(0,l)}^{(-1-\ga;\{0,l\})} \le \gd \},\\
		\Sigma^\gd  &:=  \Sigma_1^\gd \times \Sigma_1^\gd  \times  \Sigma_2^\gd   \times  \Sigma_3^\gd  \times   \Sigma_3^\gd .
	\end{cases}
\end{align}
For notational convenience, we use $\|\cdot\|_{\Sigma_i}$ to denote the norm for $ \Sigma_i$. The norm $ \|\cdot\|_\Sigma$ is understood as the sum of the norms of all components. The constraint $\sh(0)=0$ in $\Sigma_3^\gd$ keeps the triple point at the origin; it is preserved by the iteration, since the updated shock functions are obtained by integrating from $z_1=0$, cf. \eqref{eqn-stprime}.

We record the form in which the two hyperbolic invariants enter the fixed domain.

\begin{lemma}[Transport of $\BB$ and $\CA$]\label{lem-transport}
Let $h^\pm\in\Sigma_3^\gd$, set $\sigma^\pm(t):=k^\pm t+h^\pm(t)$, and let $s=s[h]$ be the extension \eqref{exp-extend}--\eqref{def-s} of the shock functions $\sigma^\pm$.
Let $\CA_{\mathrm{sh}}^\pm$ be prescribed downstream entropy traces on $S_0^\pm$, parametrized by $t\in(0,l)$, and let $\BB^I$ be the incoming Bernoulli quantity in Lagrangian coordinates.
For $\gd$ sufficiently small, the following hold.
\begin{enumerate}
\item[(a)] $s_{z_2}\ge\tfrac12$, $(\sigma^+)'\ge c>0$ and $(\sigma^-)'\le-c<0$.
The footpoint maps
\[
 \xi_s^\pm(\zz):=(\sigma^\pm)^{-1}(s(\zz)),\qquad\zz\in\DD^\pm,
\]
are well defined, with $0<\xi_s^\pm\le z_1$. Moreover,
\[
 \begin{gathered}
 C^{-1}|z_2|\le\xi_s^\pm(\zz)\le C|z_2|,\\
 C^{-1}|\zz-P^\pm|\le l-\xi_s^\pm(\zz)\le C|\zz-P^\pm|.
 \end{gathered}
\]
\item[(b)] Define $\BB_s:=\BB^I\circ s$ and $\CA_s:=\CA_{\mathrm{sh}}^\pm\circ\xi_s^\pm$ in $\DD^\pm$. Then
\begin{align}\label{def-transport}
 D_s\BB_s=D_s\CA_s=0,
 \qquad D_s:=\po_{z_1}-\frac{s_{z_1}}{s_{z_2}}\po_{z_2}.
\end{align}
This is the transport operator in the fixed domain.
\item[(c)] For $m=1,2$, write
\[
 E_m(\CA_{\mathrm{sh}}):=\sum_\pm\|\CA_{\mathrm{sh}}^\pm-\CA_0^\pm\|_{m,\ga;(0,l)}^{(-\ga;\{0,l\})}.
\]
Under the incoming-flow bound \eqref{con-UI},
\begin{align}\label{est-transport}
 &\|\BB_s-\BB_0\|_{m,\ga}^{(-\ga,-\ga)(-\ga)}
 +\|\CA_s-\CA_0\|_{m,\ga}^{(-\ga,-\ga)(-\ga)}\nonumber\\
 &\hspace{35mm}\le C\big(\ve+E_m(\CA_{\mathrm{sh}})\big).
\end{align}
\end{enumerate}
\end{lemma}

\begin{proof}
The extension estimate and $h^\pm(0)=0$ give
$s_{z_2}=1+O(\gd)$ and $(\sigma^\pm)'=k^\pm+O(\gd)$.
Since $s(z_1,0)=0$ and $s(z_1,k^\pm z_1)=\sigma^\pm(z_1)$, monotonicity yields (a).
The mean value theorem, applied to $\sigma^\pm$ and to
$s(z_1,k^\pm z_1)-s(\zz)$, gives
\[
 \begin{aligned}
 C^{-1}\big((l-z_1)+|k^\pm z_1-z_2|\big)
 &\le l-\xi_s^\pm(\zz)\\
 &\le C\big((l-z_1)+|k^\pm z_1-z_2|\big).
 \end{aligned}
\]
The sum $(l-z_1)+|k^\pm z_1-z_2|$ is bounded above and below
by fixed positive multiples of $|\zz-P^\pm|$, which proves the
second comparison in (a).
The identities in (b) follow from $D_s s=0$.
For (c), fix one of the two subdomains and omit the sign in $\sigma$, $\xi_s$ and $\CA_{\mathrm{sh}}$. Differentiating the inverse relation gives
\begin{align*}
 D\xi_s&=\frac{Ds}{\sigma'(\xi_s)},\\
 D^2\xi_s&=\frac{D^2s}{\sigma'(\xi_s)}
 -\frac{\sigma''(\xi_s)}{[\sigma'(\xi_s)]^3}Ds\otimes Ds.
\end{align*}
Let $d(\zz)$ be the minimum of $1$, the distance to $\II$,
and the distance to $P^\pm$ in the corresponding subdomain.
The segment $\II$ and the point $P^\pm$ are separated by a positive distance, so the products
of their distance weights are comparable to the corresponding
powers of $d$. Part (a) gives
$C^{-1}d\le\min(\xi_s,l-\xi_s,1)\le Cd$, and the inverse-function
formulas give
\[
 |D\xi_s|\le C,\qquad |D^2\xi_s|\le C\gd d^{\ga-1}.
\]
Consequently, for $j=1,2$,
\[
 d^{j-\ga}|D^j(\CA_{\mathrm{sh}}\circ\xi_s)|\le C E_j(\CA_{\mathrm{sh}}),
\]
using
$D^2(\CA_{\mathrm{sh}}\circ\xi_s)=\CA_{\mathrm{sh}}''(\xi_s)D\xi_s\otimes D\xi_s
+\CA_{\mathrm{sh}}'(\xi_s)D^2\xi_s$ for $j=2$.
For pairs separated by less than a fixed multiple of $d$, the
same formulas and the local scaled H\"older product inequality
give the weighted seminorm. For the other pairs, the pointwise
bounds and their separation give that seminorm directly.
Together with the zeroth-order bound this proves the entropy
part of \eqref{est-transport}. Since $\BB_0$ is constant in each subdomain, the same calculation
applied to $(\BB^I-\BB_0)\circ s$ gives the Bernoulli estimate
from \eqref{con-UI}. This proves \eqref{est-transport} for $m=1,2$.
\end{proof}

Given $V=(\wh,\ph,\rhoh,\sh^+,\sh^-)\in\GS^\gd$, solve
{\bf Problem MBV}, including the oblique condition
\eqref{con-g5}, to obtain $v_2$. In the following reconstruction,
use the original coefficients and source terms in
\eqref{def-aij}--\eqref{def-Fi}, even if their negatives were
used in the scalar existence proof. Set
\begin{align} \label{eqn-v1int}
v_1(\zz)=g^+(z_1)+\int_{k^+z_1}^{z_2}
 (F^1-a^{11}(v_2)_{z_1}-a^{21}(v_2)_{z_2})(z_1,t)\,\mathrm dt.
\end{align}
The integral crosses $\II$ using the one-sided integrands. To
justify both equations, set
\[
 Q^1=a^{11}(v_2)_{z_1}+a^{21}(v_2)_{z_2}-F^1,
 \qquad
 Q^2=a^{12}(v_2)_{z_1}+a^{22}(v_2)_{z_2}-F^2.
\]
The scalar equation gives $\partial_{z_1}Q^1+
\partial_{z_2}Q^2=0$ weakly, and hence $[Q^2]=0$ on
$\II$. Formula~\eqref{eqn-v1int} first gives
$(v_1)_{z_2}=-Q^1$. Differentiating it in $z_1$, using the
weak equation and the moving lower limit, gives
\[
 (v_1)_{z_1}=Q^2(\zz)+(g^+)'(z_1)
       +k^+Q^1(z_1,k^+z_1)-Q^2(z_1,k^+z_1)=Q^2(\zz).
\]
The last equality is exactly \eqref{con-g5}; the interface term
vanishes because $[Q^2]=0$. Thus both
\eqref{eqn-v1z1} and \eqref{eqn-v1z2} hold. The formula
also makes $v_1$ continuous across $\II$, so that
$(\wht,\pht)^\top=T^{-1}\vv$ has the required contact traces.

We first obtain the updated density trace $\rhoht|_{S_0^\pm}$ from \eqref{con-g4}.
Set $\CA_{\mathrm{sh}}^\pm(t):=(\pt/\rhot^\gc)(t,k^\pm t)$ on the shocks, using these updated traces.
Apply Lemma~\ref{lem-transport} with the \emph{input} perturbations $h^\pm=\sh^\pm$ and these $\CA_{\mathrm{sh}}^\pm$; the geometry is kept fixed until the final shock update. With the resulting $\BB_s$ and $\CA_s$, set
\begin{align}
&\rhot := \left(\tfrac{\pt}{\CA_s}\right)^{1/\gc},
&&	\ut_1 := \frac{ \sqrt{2\BB_s - \frac{2\gc \pt }{ (\gc -1 )\rhot }}}{\sqrt{1+ \wt^2}}, \label{def-Arhot}\\
& \rhoht^{\pm}:= \rhot^{\pm}- \rho_0^{\pm},&& \uht_1^{\pm}:= \ut_1^{\pm}- u_{10}^{\pm}.
\end{align}

The denominator in the shock update is uniformly separated from
zero. At the background it equals
$-u_{10}^{I\pm}w_0^{I\pm}\ne0$, and both the reconstructed states
and the incoming traces stay close to their background values.
Using \eqref{eqn-sigmaprime}, we therefore define the updated
shock functions $\st^\pm$ by
\begin{align} \label{eqn-stprime}
\st^{\pm}(z_1)=\int_0^{z_1}
 \frac{\pt^{\pm}-p^{I\pm,s}}{\ut_1^{\pm}\wt^{\pm}-u_1^{I\pm,s}w^{I\pm,s}}
 (t,k^{\pm}t)\,\mathrm dt.
\end{align}
Let $\sht^{\pm}= \st^{\pm} - s_0^{\pm}$. Define a map $$\QQ(V):=\Vt=(\wht,\pht,\rhoht,\sht^+,\sht^-).$$
{\bf Existence of a fixed point of $\QQ$.} We will prove the existence of a fixed point of $\QQ$ by using the Schauder fixed point theorem. Let $0<\ga'<\ga$. For
$V=(\wh,\ph,\rhoh,\sh^+,\sh^-)$, define
\begin{align*}
 \|V\|_{\GS'}&:=
 \|\wh\|_{2,\ga'}^{(-1-\ga',-\ga')(-\ga')}
 +\|\ph\|_{2,\ga'}^{(-1-\ga',-\ga')(-\ga')}\\
 &\quad+\|\rhoh\|_{2,\ga'}^{(-\ga',-\ga')(-\ga')}
 +\sum_\pm\|\sh^\pm\|_{3,\ga';(0,l)}^{(-1-\ga';\{0,l\})},\\
 \GS'&:=\{V:\|V\|_{\GS'}<\infty\}.
\end{align*}
The set $\GS^\gd$ is convex and compact in $\GS'$. Indeed,
away from the endpoints, the first two components are bounded
in piecewise $C^{1,\ga}$ up to $\II$, while the density
component is bounded in piecewise $C^\ga$. The compact
H\"older embeddings control these low-order traces. For the
higher derivatives, the weaker weighted norms contain the factor
$d^{\ga-\ga'}$ near $\II$; the corresponding factor near
an endpoint is $r^{\ga-\ga'}$. The shock functions and
their first derivatives are equicontinuous up to the endpoints,
and their higher derivatives are treated with the same
weighted argument. On the remaining compact subsets use the
ordinary compact embedding. A diagonal subsequence therefore
converges in $\GS'$, and the bounds defining $\GS^\gd$ pass to the limit
by lower semicontinuity.
 We will show that for suitably small $\gd$ and $\ve$, $\QQ$ is a continuous map from $\GS^\gd$ to itself.

Lemma \ref{lem-apri} and \eqref{def-vv}, \eqref{eqn-v1int} imply that
\begin{align}
	\|(\wht, \pht)	\|_{\GS_1}  =  	\|T
	^{-1}\vv	\|_{\GS_1} \le C  (\gk+|g^+(0)|),
\end{align}
where $\gk$ is given by \eqref{def-kappa}.

To estimate the source terms in the weighted $C^{1,\ga}$ norm,
we differentiate the remainders. The definitions \eqref{def-f1}, \eqref{def-f2} and
\eqref{def-Fi} give
\[
 |F|\le C\big(|D\sh|+|\Ubh|+|\BB_s-\BB_0|\big)
                     (|D\wh|+|D\ph|).
\]
For example, near $\II$ away from $P^\pm$, restrict to $|z_2|\le1$ and write
$d=|z_2|$ and $R=\min(|\zz|,|\zz-P^0|,1)$.
Then $d\le R$ and the norms give
\[
 |D(w,p)|\le C\gd R^{\ga-1},\quad
 |D^2(w,p)|\le C\gd R^{-1}d^{\ga-1},\quad
 |D\rho|\le C\gd d^{\ga-1}.
\]
Together with \eqref{est-extend}, these bounds control each
term in $DF$, with weight $d^{1-\ga}R$; the H\"older
product inequality gives the corresponding seminorm. The
corner weights at $P^\pm$ are handled in the same way.
Consequently,
\begin{align*}
 \sum_i\|F^i\|_{1,\ga}^{(-\ga,1-\ga)(1-\ga)}
 &\le C(\|V\|_\Sigma+\ve)\|V\|_\Sigma,\\
 \sum_{\sigma=+,-}\|g^\sigma\|_{2,\ga;(0,l)}^{(-\ga;\{0,l\})}
 &\le C(\|V\|_\Sigma^2+\ve).
\end{align*}
The latter estimate follows by Taylor expansion of $G_i$ at
the background, with the incoming traces treated as data.
Taking one tangential derivative in \eqref{def-g5}, and using
the trace norm of $F$, gives the required $C^{1,\ga}$ weighted
bound for $g_5$.

 It follows that
\begin{align}  \label{est-kappa}
\gk & \le C (\|V\|_{\GS}^2 +\ve)  \le C (\gd^2 +\ve) ,    \qquad\text{by }\eqref{con-UI}.
\end{align}
The same Taylor estimate gives $|g^+(0)|\le C(\gd^2+\ve)$.
The bound for $v_1$ uses the first-order identities just proved:
they control its first and second derivatives directly, without
differentiating \eqref{eqn-v1int} twice. The gradients are bounded by $Cr^{\ga-1}$ near each corner and are
therefore integrable along radial segments, so integration from the
shock trace also controls the corner values and oscillations.

The density estimate uses Lemma~\ref{lem-transport} with $m=2$.

The Taylor remainder $g_4$ in \eqref{con-g4} has no linear term in the input perturbation when the incoming data are fixed at the background. Its weighted shock-trace norm therefore satisfies
\[
 \|g_4\|_{2,\ga;(0,l)}^{(-\ga;\{0,l\})}
 \le C(\|V\|_\Sigma^2+\ve).
\]
Using \eqref{con-g4}, the trace estimate for $(\wht,\pht)$, and the smooth algebraic map $(p,\rho)\mapsto p/\rho^\gc$ gives
\begin{align*}
 E_2(\CA_{\mathrm{sh}})&\le C\big(\ve+\|V\|_\Sigma^2+\|(\wht,\pht)\|_{\Sigma_1}\big)
 \le C(\gd^2+\ve),\\
 \|\rhoht\|_{\Sigma_2}&\le C\big(\|\pht\|_{\Sigma_1}+\ve+E_2(\CA_{\mathrm{sh}})\big)
 \le C(\gd^2+\ve).
\end{align*}
All shock traces are taken on $S_0^\pm$, where the domain weights give the corresponding endpoint weights. The algebraic reconstruction of $\ut_1$ and the integration in \eqref{eqn-stprime} give the same bound for the remaining components. Consequently,
\begin{align}\|\QQ(V)\|_{\GS} & \le   C (\gd^2 +\ve) \le C_0 \ve,
\end{align}
if $\gd^2 \le \ve$.

Let $\gd_*>0$ be small enough for all preceding geometric,
positivity and reconstruction estimates, and increase $C_0$ to
at least one if necessary. Choose
\[
 0<\ve\le\ve_0<\min\{C_0^{-2},\gd_*/C_0\},
 \qquad \gd=C_0\ve.
\]
Then $\gd\le\gd_*$ and $\gd^2\le\ve$, so the preceding
estimate gives $\QQ(\GS^\gd)\subset\GS^\gd$.

We use the following compactness argument to show the continuity of $\mathcal{Q}$.

Let $V^n\to V^0$ in $\GS'$, with $V^n,V^0\in\GS^\gd$.
Compactness gives, from every subsequence of $\QQ(V^n)$, a
further subsequence converging in $\GS'$ to some $\Vt^0$.
The shock extensions converge in the weaker weighted norm by
\eqref{est-extend}, applied with $\ga'$. Their derivatives
remain uniformly nondegenerate, so the inverse shock maps and
the footpoint maps converge locally together with the
derivatives needed in the reconstruction. The incoming traces
and the functions $g_i,F^i$ therefore converge on compact
subsets of each subdomain and separately up to the interior of
$\II$; the boundary data converge away from the vertices.
The uniform weighted bounds also give convergence of the
continuous corner values.

Passing to the weak equation for $v_2$ preserves the flux
condition on $\II$, and passing to the one-sided boundary
equations preserves the oblique and Dirichlet conditions.
The uniqueness in Lemma~\ref{lem-exist} identifies the scalar
limit with the solution associated with $V^0$. Formula
\eqref{eqn-v1int} then identifies $v_1$, and the transport and
algebraic formulas identify the density. Finally, the shock
integrands in \eqref{eqn-stprime} converge uniformly and their
denominators stay bounded away from zero, so their integrals
give the shock components of $\QQ(V^0)$. Hence
$\Vt^0=\QQ(V^0)$. Every convergent subsequence has this same
limit, which proves continuity of $\QQ$ in $\GS'$.
The Schauder fixed point theorem now gives a fixed point.

\smallskip

{\bf Uniqueness of a fixed point of $\QQ$.}
Let $V_1,V_2\in\Sigma^\gd$ be two fixed points for the same incoming flow satisfying \eqref{con-unique-UI}, and write $\check f:=f_2-f_1$.
The density difference is measured with one derivative less in the interior, but its second-order shock traces must also be retained. Set
\begin{align*}
 N_{wp}&:=\|\check w\|_{\Sigma_1}+\|\check p\|_{\Sigma_1},
 &N_s&:=\sum_\pm\|\check s^\pm\|_{\Sigma_3},\\
 N_\rho&:=\|\check\rho\|_{1,\ga}^{(-\ga,-\ga)(-\ga)},
 &N_{\rho,S}&:=\sum_\pm\|\check\rho|_{S_0^\pm}\|_{2,\ga;(0,l)}^{(-\ga;\{0,l\})}.
\end{align*}
These quantities are finite by the estimates already proved. Write
$\sigma_i^\pm(t)=k^\pm t+h_i^\pm(t)$, $s_i=s[h_i]$, and
$\xi_i^\pm=(\sigma_i^\pm)^{-1}\circ s_i$.

\emph{Transport differences.}
Both the inverse shock functions and the extensions vary with
the geometry. Near $\II$, $s_i=(1+q_{i,\pm}/k^\pm)z_2$;
the corresponding footpoint maps must therefore both be varied.

Using $\sigma_i(\xi_i)=s_i$, we obtain
\[
 \sigma_2(\xi_2)-\sigma_2(\xi_1)
 =\check s(\zz)-\check\sigma(\xi_1).
\]
Near the triple point, the extension satisfies
\[
 |\check s(\zz)|\le C N_s|z_2|,\qquad
 |\check\sigma(\xi_1)|\le C N_s\xi_1.
\]
In the conical
neighbourhood of $\II$ the first bound follows from its linear
formula. On the support of the angular cutoff,
$\tfrac12|k^\pm|z_1\le|z_2|\le|k^\pm|z_1$, so the extension
estimate gives the same bound. By part (a) of Lemma~\ref{lem-transport},
$C^{-1}|z_2|\le\xi_1\le C|z_2|$. Thus both terms are bounded by
$C N_s\xi_1$.

Near $P^\pm$, use $\check s(z_1,k^\pm z_1)=\check\sigma(z_1)$ and the uniform Lipschitz bound for $\check s$ to obtain
\[
 |\check s(\zz)-\check\sigma(\xi_1)|
 \le C N_s\big(|k^\pm z_1-z_2|+|z_1-\xi_1|\big)
 \le C N_s(l-\xi_1).
\]
The extension estimates therefore give
\[
 |\xi_2-\xi_1|\le C N_s\min(\xi_1,l-\xi_1),
\]
and their differentiated, weighted versions follow from the inverse-function formulas in Lemma~\ref{lem-transport}.
Let $\CA_{\mathrm{sh},i}^\pm$ be the entropy trace associated with $V_i$. Suppressing the sign, we have
\[
 \CA_{\mathrm{sh},2}\circ\xi_2-\CA_{\mathrm{sh},1}\circ\xi_1
 =(\CA_{\mathrm{sh},2}-\CA_{\mathrm{sh},1})\circ\xi_2+
 \int_0^1 \CA_{\mathrm{sh},1}'(\xi_1+t(\xi_2-\xi_1))\,\mathrm dt\,(\xi_2-\xi_1).
\]
The bound $E_2(\CA_{\mathrm{sh},1})\le C\gd$ controls the first derivatives of the last expression. Weighted product and composition estimates yield
\begin{align*}
 \|\CA_{\mathrm{sh},1}\circ\xi_2-\CA_{\mathrm{sh},1}\circ\xi_1\|_{1,\ga}^{(-\ga,-\ga)(-\ga)}&\le C\gd N_s,\\
 \|\check\CA_s\|_{1,\ga}^{(-\ga,-\ga)(-\ga)}&\le C(N_{wp}+N_{\rho,S}+\gd N_s).
\end{align*}
For the incoming quantities, the ordinary mean-value formula
\begin{align*}
\check U^{I\pm}|_{S_0^\pm}(z_1)
 &=U^{I\pm}(z_1,\sigma_2^\pm(z_1))
       -U^{I\pm}(z_1,\sigma_1^\pm(z_1))\\
 &=\int_0^1\partial_{y_2}U^{I\pm}
       (z_1,\sigma_1^\pm(z_1)+t\check\sigma^\pm(z_1))
       \,\mathrm dt\,\check\sigma^\pm(z_1)
\end{align*}
and \eqref{con-unique-UI} give
\begin{align}
 \sum_\pm\|\check U^{I\pm}|_{S_0^\pm}\|_{2,\ga;(0,l)}^{(-\ga;\{0,l\})}
 &\le C\ve N_s,\label{est-gdUI}\\
 \|\check\BB_s\|_{1,\ga}^{(-\ga,-\ga)(-\ga)}
 +\sum_\pm\|\check\BB_s|_{S_0^\pm}\|_{2,\ga;(0,l)}^{(-\ga;\{0,l\})}
 &\le C\ve N_s.\label{est-gdBB}
\end{align}
The density reconstruction consequently implies
\begin{equation*}
 N_\rho\le C(N_{wp}+N_{\rho,S})+C(\gd+\ve)N_s.
\end{equation*}

\emph{Elliptic and shock-trace differences.}
The interior terms $F^i$ require only first derivatives of the density difference, whereas the $C^{2,\ga}$ Dirichlet data require its second-order shock trace.
Applying Lemma~\ref{lem-apri} and the reconstruction \eqref{eqn-v1int}, including the integration constant $\check g^+(0)$, gives
\begin{align*}
 N_{wp}&\le C(\gd+\ve)(N_{wp}+N_\rho+N_{\rho,S}+N_s),\\
 |\check g^+(0)|&\le C(\gd+\ve)(N_{wp}+N_{\rho,S}+N_s).
\end{align*}
Indeed the coefficients of the linearized remainders vanish at the background, and their weighted product estimates give the factor $\gd+\ve$.
For the traces, \eqref{con-g4} yields
\begin{align*}
 N_{\rho,S}&\le C(N_{wp}+\ve N_s)+C(\gd+\ve)N_{\rho,S}.
\end{align*}
At a fixed point, $G_1=0$ on each shock and $\partial_\rho G_1=b_{13}+O(\gd+\ve)\ne0$. Thus the implicit function theorem expresses the density trace as a smooth function of the pressure, direction and incoming traces. Equivalently, $g_4$ is a quadratic downstream remainder plus the incoming-data perturbation; its difference is estimated on the fixed shocks, so no second interior derivative of $\check\rho$ is used.
Finally, applying \eqref{eqn-stprime} at the two fixed points, using the smooth algebraic reconstruction of $u_1$ on the shocks and \eqref{est-gdUI}--\eqref{est-gdBB}, gives
\[
 N_s\le C(N_{wp}+N_{\rho,S})+C\ve N_s.
\]
All constants depend only on the fixed background and $\ga$.
For $\gd+\ve$ sufficiently small, the last two estimates first imply
$N_{\rho,S}+N_s\le C N_{wp}$. The transport estimate then gives $N_\rho\le C N_{wp}$, and hence
\[
 N_{wp}\le C(\gd+\ve)N_{wp}\le\tfrac12N_{wp}.
\]
Thus $N_{wp}=N_{\rho,S}=N_s=N_\rho=0$, so the two fixed points coincide.

\smallskip

\noindent {\bf Proof of Theorem~\ref{thm2}.}
By the fixed point argument above, there exists a fixed point
$V=(\wh,\ph,\rhoh,\sh^+,\sh^-)\in\GS^\gd$ of the map $\QQ$, which gives
a solution $(U,s^\pm)$ to {\bf Problem B}. At a fixed point,
the definitions of $f_1,f_2$ recover the elliptic subsystem for $w,p$;
the equations for $g_i$ give $G_i=0$ on the shocks, and
\eqref{eqn-stprime} recovers their slopes. Since the mass
flux and pressure jumps remain nonzero near the background,
these relations recover all the shock conditions.
The reconstruction transports $\BB$ and $\CA$ in the
Lagrangian coordinates. These two transport identities and
the elliptic system are equivalent to
\eqref{eqn-euler1}--\eqref{eqn-euler4} in each subdomain.
Continuity of $v_1,v_2$ gives continuity of $w,p$ on $\II$.
Positivity, strict subsonicity, and the strict entropy
inequalities at the shocks persist by the smallness of
the perturbation. From the estimates
\eqref{est-kappa} and $\|\QQ(V)\|_{\GS}\le C_0\ve$, together with the
definitions of the norms \eqref{def-UGO}, \eqref{def-fGO} and the
reconstruction formulas \eqref{def-Arhot}, we obtain the estimate
\eqref{est-U}, where the constant $C$ depends only on $\Xi$ and $\ga$. The
uniqueness statement follows from the uniqueness of the fixed point
established above. $\square$

\smallskip

\noindent {\bf Proof of Theorem~\ref{thm1}.}
Let $(U, s^\pm)$ be the solution to {\bf Problem B} given by
Theorem~\ref{thm2}, with physical state $U=(\rho,u_1,u_2,p)^\top$
and corresponding variables $\CU=(\rho,u_1,w,p)^\top$ in the Lagrangian
coordinates $(y_1,y_2)$. We first construct the inverse coordinate
transformation. Set
\begin{align}
\label{def-zeta}	\gz(y_1)&:=\int_0^{y_1}w(t,0)\,\mathrm dt,\\
	X(y_1,y_2)&:=\gz(y_1)+\int_0^{y_2}\frac{1}{\rho u_1}(y_1,t)\,\mathrm dt ,\label{def-X}
\end{align}
the second integral being taken in $\GO^+$ when $y_2>0$ and in $\GO^-$ when $y_2<0$. We claim that
\begin{align}\label{est-X}
	X_{y_1}=w,\qquad X_{y_2}=\frac{1}{\rho u_1}\qquad\text{in }\GO^\pm .
\end{align}
The second identity is immediate. For the first, by \eqref{eqn-euler1} one has $(1/\rho u_1)_{y_1}=w_{y_2}$, so that
$$
X_{y_1}(y_1,y_2)=\gz'(y_1)+w(y_1,y_2)-w(y_1,0)=w(y_1,y_2),
$$
where we used $\gz'(y_1)=w(y_1,0)$ and the continuity of $w$ across $\II$, cf. \eqref{con-w}. Since $c\le\rho u_1\le C$, the map $y_2\mapsto X(y_1,y_2)$ is strictly increasing, so
$$
\Phi^{-1}:\ (y_1,y_2)\mapsto (x_1,x_2)=(y_1,X(y_1,y_2))
$$
is a bi-Lipschitz homeomorphism from $\GO$ onto its image. Its derivatives satisfy
$(X_{y_1},X_{y_2})=(w,1/\rho u_1)$ in each of $\GO^+$ and $\GO^-$ and have one-sided traces on $\II$.
The tangential derivative $X_{y_1}=w$ is continuous across $\II$, whereas the transverse derivative $X_{y_2}$ may jump; no global $C^1$ regularity is asserted. Defining $U(\xx):=U(\Phi(\xx))$ for $\xx=\Phi^{-1}(\yy)$, we get $\psi(x_1,x_2)=y_2$ with $\psi_{x_2}=\rho u_1$ and $\psi_{x_1}=-\rho u_2$; the equivalence of \eqref{eqn-euler1}--\eqref{eqn-euler4} with the Euler system in the $\xx$-variables shows that $U$ solves \eqref{Euler1}--\eqref{Euler3} in the two downstream regions. The continuity of $w,p$ gives the physical contact conditions. To recover the shock conditions with the prescribed physical incoming states, we must also check that the two coordinate constructions agree on each shock. The images of $S^\pm$ and of $\II$ are the shock curves and the slip line in the $\xx$--variables.

We next check that the $\xx$-coordinates of the shocks are the same as those of the incoming flow. Let $X^{I\pm}$ be the inverse of the incoming stream-function transformation in each incoming region, with $X^{I\pm}(0,0)=0$.
It is constructed within that region from mass conservation and satisfies
$X^{I\pm}_{y_1}=w^{I\pm}$ and $X^{I\pm}_{y_2}=1/(\rho^{I\pm}u_1^{I\pm})$;
no extension of the incoming state to the downstream slip line is required. Then
$$
\frac{\dx}{\dx y_1}\Big(X\big(y_1,s^\pm(y_1)\big)-X^{I\pm}\big(y_1,s^\pm(y_1)\big)\Big)
=[w]+(s^\pm)'\Big[\frac{1}{\rho u_1}\Big]=0
$$
by \eqref{con-RH1}, so the two transformations agree on the shocks, since they agree at $y_1=0$; hence the two coordinate constructions give the same physical shock curves. The slip line is the image of $\II$ under $\Phi^{-1}$, and the estimates \eqref{est-U} and \eqref{est-zeta} transfer as explained below.

The bi-Lipschitz property makes the distances to the contact
curve and its endpoints comparable in the two coordinate
systems. Write $m=(\rho u_1)^{-1}$. Then
$DX=(w,m)$ and $D^2X=D(w,m)$ on each side. Near the
interior of the contact, $Dm$ has at most the growth
$d^{\ga-1}$, while $D(w,p)$ is bounded. Thus the second
chain-rule term for $(w,p)$ has the allowed growth
$d^{\ga-1}$; for $(u_1,\rho)$ its worst product has growth
$d^{2\ga-2}\le d^{\ga-2}$ for $d\le1$.
The same chain-rule calculation with the endpoint weights,
and the corresponding H\"older product estimates, transfers
the stated flow norms in both directions.

For completeness, denote the physical shock by
$s_x^\pm(t):=X(t,s^\pm(t))$. Along the shock,
\[
 (s_x^\pm)'=w(t,s^\pm(t))
       +(s^\pm)'(t)m(t,s^\pm(t)).
\]
The right-hand side differs from its background value by a
function bounded in $C^{2,\ga}_{(-\ga;\{0,l\})}$.
Integration from zero therefore gives
$s_x^\pm-s_{x,0}^\pm\in
C^{3,\ga}_{(-1-\ga;\{0,l\})}$ with the bound in
\eqref{est-U1}. This verifies the shock estimate in the
physical coordinates as well as the flow estimate.

By the definition of the slip line $\GC$ in \eqref{def-zeta} and the regularity of $w$ given in \eqref{est-U}, it follows that
\begin{align}
\|\gz\|_{2,\ga;(0,l)}^{(-\ga-1;\{0,l\})}
\le C\,\|U^I-U^I_0\|_{2,\ga;\GO^I_0}. \label{est-zeta}
\end{align}
On the truncation segment $\LL$, condition \eqref{con-truncy} is equivalent to
condition \eqref{con-trunc}. The estimate \eqref{est-U} yields the estimate
\eqref{est-U1}, with constants depending on $\Xi$ and $\ga$. Finally,
the uniqueness of the solution to {\bf Problem A} follows from the
one-to-one correspondence between {\bf Problem A} and {\bf Problem B}
established by the coordinate transformation and the uniqueness in
Theorem~\ref{thm2}. $\square$

\section*{Acknowledgments}
The research of Jun Chen and Xuemei Deng was supported in part by the
Hubei Provincial Natural Science Foundation of China (Grant No. 2025AFB530).
The work of Xiaoguang You was partially supported by the National
Natural Science Foundation of China (Grant No. 12561038) and the Jiangxi
Provincial Natural Science Foundation (Grant No. 20242BAB25008).


\begin{thebibliography}{00}

\bibitem{Bae} M. Bae,
Stability of contact discontinuity for steady Euler system in infinite duct,
{\it Z. Angew. Math. Phys.}, {\bf 64} (2013), 917--936.

\bibitem{BCF} M. Bae, G.-Q. Chen and M. Feldman,
Regularity of solutions to regular shock reflection for potential flow,
{\it Invent. Math.}, {\bf 175} (2009), 505--543.

\bibitem{BCF1} M. Bae, G.-Q. Chen and M. Feldman,
Prandtl--Meyer reflection for supersonic flow past a solid ramp,
{\it Quart. Appl. Math.}, {\bf 71} (2013), 583--600.

\bibitem{BCF2} M. Bae, G.-Q. Chen and M. Feldman,
{\it Prandtl--Meyer Reflection Configurations, Transonic Shocks, and Free Boundary Problems},
Memoirs of the American Mathematical Society, Vol. 301, No. 1507,
American Mathematical Society, Providence, RI, 2024.

\bibitem{BP} M. Bae and H. Park,
Contact discontinuities for 2-dimensional inviscid compressible flows in infinitely long nozzles,
{\it SIAM J. Math. Anal.}, {\bf 51} (2019), 1730--1760.

\bibitem{Busemann} A. Busemann,
{\it Gasdynamik}, Handbuch der Experimentalphysik, Vol. IV,
Akademische Verlagsgesellschaft, Leipzig, 1931.

\bibitem{CCF1} G.-Q. Chen, J. Chen and M. Feldman,
Stability and asymptotic behavior of transonic flows past wedges for the full Euler equations,
{\it Interfaces Free Bound.}, {\bf 19} (2017), 591--626.

\bibitem{CCX} G.-Q. Chen, J. Chen and W. Xiang,
Stability of attached transonic shocks in steady potential flow past three-dimensional wedges,
{\it Comm. Math. Phys.}, {\bf 387} (2021), 111--138.

\bibitem{CF10} G.-Q. Chen and M. Feldman,
Global solutions of shock reflection by large-angle wedges for potential flow,
{\it Ann. of Math.}, {\bf 171} (2010), 1067--1182.

\bibitem{CF18} G.-Q. Chen and M. Feldman,
{\it The Mathematics of Shock Reflection--Diffraction and von Neumann's Conjectures},
Annals of Mathematics Studies, Vol. 197, Princeton University Press, Princeton, NJ, 2018.

\bibitem{CFX} G.-Q. Chen, M. Feldman and W. Xiang,
Convexity of self-similar transonic shocks and free boundaries for the Euler equations for potential flow,
{\it Arch. Ration. Mech. Anal.}, {\bf 238} (2020), 47--124.

\bibitem{CHWX} G.-Q. Chen, F.-M. Huang, T.-Y. Wang and W. Xiang,
Steady Euler flows with large vorticity and characteristic discontinuities in arbitrary infinitely long nozzles,
{\it Adv. Math.}, {\bf 346} (2019), 946--1008.

\bibitem{ADV1} J. Chen,
Stability of vortex lines attached to airfoils in subsonic Euler flows,
{\it Adv. Math.}, {\bf 503} (2026), 111237.

\bibitem{ChenDeng} J. Chen and X. Deng,
Elliptic equations in divergence form with discontinuous coefficients in domains with corners,
{\it Acta Math. Sci. Ser. B}, {\bf 44} (2024), 1903--1915.

\bibitem{CXZ} J. Chen, Z. Xin and A. Zang,
Subsonic flows past a profile with a vortex line at the trailing edge,
{\it SIAM J. Math. Anal.}, {\bf 54} (2022), 912--939.

\bibitem{ChenSX1} S.-X. Chen,
Stability of a Mach configuration,
{\it Comm. Pure Appl. Math.}, {\bf 59} (2006), 1--35.

\bibitem{ChenSX2} S.-X. Chen,
Mach configuration in pseudo-stationary compressible flow,
{\it J. Amer. Math. Soc.}, {\bf 21} (2008), 63--100.

\bibitem{ChenSX3} S.-X. Chen,
E-H type Mach configuration and its stability,
{\it Comm. Math. Phys.}, {\bf 315} (2012), 563--602.

\bibitem{ChenSX4} S.-X. Chen,
Global existence and stability of a stationary Mach reflection,
{\it Sci. China Math.}, {\bf 58} (2015), 11--34.

\bibitem{CF} R. Courant and K. O. Friedrichs,
{\it Supersonic Flow and Shock Waves},
Interscience Publishers, New York, 1948.

\bibitem{EllingLiu} V. Elling and T.-P. Liu,
Supersonic flow onto a solid wedge,
{\it Comm. Pure Appl. Math.}, {\bf 61} (2008), 1347--1448.

\bibitem{FangLiuYuan} B. Fang, L. Liu and H. Yuan,
Global uniqueness of transonic shocks in two-dimensional steady compressible Euler flows,
{\it Arch. Ration. Mech. Anal.}, {\bf 207} (2013), 317--345.


\bibitem{GH} D. Gilbarg and L. H\"ormander,
 Intermediate Schauder estimates,
{\it Arch. Ration. Mech. Anal.}, \textbf{74} (1980),
297--318.

\bibitem{gt} D. Gilbarg and N. S. Trudinger,
{\it Elliptic Partial Differential Equations of Second Order},
2nd ed., Springer-Verlag, Berlin, 1983.


\bibitem{LiNirenberg} Y. Li and L. Nirenberg,
Estimates for elliptic systems from composite material,
{\it Comm. Pure Appl. Math.}, {\bf 56} (2003), 892--925.

\bibitem{LiVogelius} Y.  Li and M. Vogelius,
Gradient estimates for solutions to divergence form elliptic equations
with discontinuous coefficients,
{\it Arch. Ration. Mech. Anal.}, {\bf 153} (2000), 91--151.

\bibitem{Lieberman1} G. M. Lieberman,
The Perron process applied to oblique derivative problems,
{\it Adv. Math.}, {\bf 55} (1985), 161--172.

\bibitem{Lieberman2} G. M. Lieberman,
Mixed boundary value problems for elliptic and parabolic differential equations of second order,
{\it J. Math. Anal. Appl.}, {\bf 113} (1986), 422--440.

\bibitem{Mach} E. Mach,
\"Uber den Verlauf von Funkenwellen in der Ebene und im Raume,
{\it Sitzungsber. Akad. Wiss. Wien}, {\bf 78} (1878), 819--838.

\bibitem{Meyer} Th. Meyer,
{\it \"Uber zweidimensionale Bewegungsvorg\"ange in einem Gas, das mit \"Uberschallgeschwindigkeit str\"omt},
Dissertation, G\"ottingen, 1908; Forschungsheft des Vereins deutscher Ingenieure,
Vol. 62, Berlin, 1908, 31--67.

\bibitem{Prandtl} L. Prandtl,
Allgemeine \"Uberlegungen \"uber die Str\"omung zusammendr\"uckbarer Fl\"ussigkeiten,
{\it Z. Angew. Math. Mech.}, {\bf 16} (1936), 129--142.

\bibitem{Serre} D. Serre,
Shock reflection in gas dynamics, in: {\it Handbook of Mathematical Fluid Dynamics}, Vol. 4,
Elsevier, Amsterdam, 2007, 39--122.

\bibitem{vonNeumann} J. von Neumann,
Oblique reflection of shocks, in: {\it Collected Works}, Vol. 6,
Pergamon Press, Oxford, 1963, 238--299.

\end{thebibliography}
\end{document}